\documentclass[a4paper, 12pt, english, notitlepage]{article}
\usepackage[top=3cm, bottom=3cm, left=2.5cm, right=2.5cm]{geometry}

\usepackage[utf8]{inputenc} % Required for inputting international UTF-8 characters
\usepackage[T1]{fontenc} % Output font encoding for international characters
\usepackage{babel} % Language package (language is defined in documentclass)
\usepackage{setspace} % Font/line (and others) spacing configuration
\usepackage[section]{placeins} % Place figures in their section only

\usepackage{graphicx} % Import and manage graphics
\usepackage{tabularx} % Better table control (ie. auto cell's width)
\usepackage{multirow} % Allows multirow cells in tables
\usepackage{pdfpages} % Insert PDF documents
\usepackage{subcaption} % To use subfigures
\usepackage[export]{adjustbox} % Adjust figures position
\usepackage{marginnote} % Insert notes in doc margins
\usepackage{tikz}

\newcolumntype{\grayc}{>{\itshape\large\raggedright\arraybackslash\columncolor{gray90}}c}
\newcolumntype{\centerX}{>{\centering\arraybackslash}X}

\usepackage{listings} % Create lists with line numbers (used for code view)
\usepackage{xcolor,colortbl} % Table cells color editing
\usepackage{color} % Main color package (define colors)
\definecolor{accentcolor}{RGB}{0,119,139} % Accent color used in the whole doc
\definecolor{gray75}{gray}{0.75} % Title lines main color
\definecolor{gray90}{gray}{0.90} % Table main rows/columns
\definecolor{darkgreen}{rgb}{.0, .5, .0} % EXTENDS command color

\usepackage[colorlinks = true,
            linkcolor = red,
            urlcolor  = black,
            citecolor = red,
            anchorcolor = black]{hyperref}
\usepackage{titlesec} % Lets modify header styles (chapter, section, etc)
\newcommand{\hspct}{\hspace{18pt}} % Horizontal space for chapter title
\newcommand{\hspst}{\hspace{9pt}} % Horizontal space for section title
\newcommand{\hspsst}{\hspace{4.5pt}} % Horizontal space for subsection title

\titleformat{\chapter}[hang]
    {\Huge\bfseries\raggedright} % Huge size and bold font
    {\thechapter\hspct % Number and horizontal space
     \textcolor{gray75}{\rule[-0.25\baselineskip]{5pt}{\baselineskip}}
     \hspct}{0pt}{} % Horizontal space then name

\titleformat{\section}[hang]
    {\Large\bfseries\raggedright} % Huge size and bold font
    {\thesection\hspst\textcolor{gray75}{-}\hspst}{0pt}{} % Number, horizontal space then name
    [\hrule height 0.5pt]

\titleformat{\subsection}[hang]
    {\large\bfseries\raggedright} % Huge size and bold font
    {\thesubsection\hspsst\textcolor{gray75}{-}\hspsst}{0pt}{} % Number, horizontal space then name

\usepackage{amsmath} % Create matrices, subequations and more
\usepackage{amssymb} % Math font to create symbols from characters

\iflanguage{french}{ % Replace "Table" with "Tableau" in FRENCH only
    \AtBeginDocument{}
}
    
\makeatletter
\newcommand*{\field}[1]{\gdef\@field{#1}} % Defines "field" variable used in title page
\newcommand*{\project}[1]{\gdef\@project{#1}} % Defines "project" variable used in title page
\newcommand*{\mainlogo}[1]{\gdef\@mainlogo{#1}} % Defines "mainlogo" variable used in title page
\newcommand*{\optionallogo}[1]{\gdef\@optionallogo{#1}} % Defines "optionallogo" variable used in title page
\makeatother

\usepackage{etoolbox} % Allows to append commands to existing environements

\iflanguage{english}{
    \AtBeginEnvironment{itemize}{\setlength{\baselineskip}{0em}} % Change itemize line skip
    \AtBeginEnvironment{itemize}{\setlength{\parskip}{.7em}} % Change itemize line vertical space
    
    \AtBeginEnvironment{enumerate}{\setlength{\baselineskip}{0em}} % Change enumerate line skip
    \AtBeginEnvironment{enumerate}{\setlength{\parskip}{.7em}} % Change enumerate line vertical space
}

\AtBeginEnvironment{tabularx}{} % Change array vertical padding within cells
\usepackage{preamble}
    \title{Morse index stability for Yang-Mills connections}
    \author{Mario Gauvrit, Paul Laurain}
    \date{}
\begin{document}
 \hypersetup{
    linkcolor=blue,
    filecolor=magenta,      
    urlcolor=blue
    }

  \maketitle
  
  \begin{abstract}  We  prove stability results of the Morse index plus nullity of Yang-Mills connections in dimension 4 under weak convergence. Precisely we establish that the sum of the Morse indices and the nullity of a bounded sequence of Yang-Mills connections is asymptotically bounded above by the sum of the Morse index and the nullity of the weak limit and the bubbles while the Morse indices are  asymptotically bounded below by the sum of the Morse index of the weak limit and the bubbles. 
  \end{abstract}
  
  \section{Introduction}
	
Yang–Mills theory for $4$-manifolds is often associated to the celebrated    result of  Donaldson \cite{Donaldson}, see also \cite{FU} for a clear exposition, who, using the moduli spaces of anti-self-dual connections, get topological obstruction on the topology of smooth $4$-manifolds. Self-dual and anti-self-dual connections are known to be weakly stable. Reciprocally, on $S^4$, Bourguignon and Lawson \cite{BourguignonLawson} have proved that any weakly stable connection is either self-dual or anti-self-dual. There are very few examples of non self-dual (or anti-self-dual) connections, and then non-stable connections, on a general manifold. We must mention the landmark paper of Sibner and Uhlenbeck \cite{SibnerU} which gives such examples over $S^4$. So the goal of this paper consists in studying the stability of the index of general Yang-Mills connections.\\
  
  In the following, we will use the bundle-free description of weak connections introduced by Petrache and Rivière in \cite{PETRACHE2017469}, see also section VIII.1 \cite{rivière2015variations}. We develop this notion and the appropriate notion of convergence in the appendix \ref{weak}. In particular we prove that this notion is strictly equivalent to the notion of $\mathrm{W}^{1,2}$-connection when the topology is trivial, see lemma \ref{globalg}. We think that the framework developed in the appendix is the appropriate one to do analysis on connections in the critical dimension (dimension $4$). For instance, thanks to this framework, the bubble tree convergence is also state without the explicit use of gauges, see theorem \ref{bubbletree}.\\ 
  
  Let $(M^4, h)$ be an arbitrary Riemannian manifold, $G$ a compact Lie subgroup of $\mathrm{SU}(n)$, with Lie algebra $\mathfrak{g}$. In particular, by theorem A of \cite{PR14}, every $\mathrm{W}^{2,2}$-bundle is a trivial $\mathrm{W}^{1,(4,\infty)}$-bundle, see also the appendix. Here $\mathrm{L}^{4,\infty}$ is the weak $\mathrm{L}^4$ space, see \cite{Grafakos1} for a complete introduction to Lorentz's spaces. Using one such trivialization, every $\mathrm{W}^{1,2}$-connection on a $\mathrm{W}^{2,2}$ G-bundle over $M$ correspond to a $\mathfrak{g}$-valued 1-form $A$ (called the connection form) such that 
  \begin{itemize}
      \item $A\in \mathrm{L}^{4,\infty}(M,T^*M\otimes \mathfrak{g}) $,
      \item $F_A= \diff A + A \wedge A$ (defined as a distribution) is in $\mathrm{L}^2$,
      \item locally, there exists a $\mathrm{W}^{1,(4,\infty)}$-gauge $g$ satisfying $A^g\in \mathrm{W}^{1,2}$, where $ A^g := g^{-1} A g+g^{-1}\diff g$ is the expression of $A$ after the gauge change $g$.
  \end{itemize} 
 
  Denote by $\mathfrak U_G\left(M^4\right)$ the space of such connections.  Since $\mathrm{L}^{4,\infty}(M,T^*M\otimes \mathfrak{g}) \subset \mathrm{L}^2(M,T^*M\otimes \mathfrak{g})$, you can also consider $A\in \mathrm{L}^2$ with $\mathrm{W}^{1,2}-$gauge. Note that the set of connections on a fixed bundle can be realized as an affine space contained in $\mathfrak U_G(M^4)$ with direction $\mathrm{W}^{1,2}(M, T^*M\otimes \mathfrak{g})$. The total curvature, also called the Yang-Mills energy, is defined for $A\in\mathfrak U_G\left(M^4\right)$ by $$\mathcal{YM}(A)=\frac{1}{2}\int_M\left|F_{A}\right|_h^2 \mathrm{vol}_h = \frac{1}{2}\int_M \tr(F_A \wedge \star F_A).$$  This quantity is unchanged by conformal transformation (a specific property of dimension $4$) and gauge transformation: if $g\in \mathrm{W}^{2,2}(M^4,G)$ then $A^g\in\mathfrak U_G\left(M^4\right)$, $|F_{A^g}|_h = |F_A|_h$ and $\mathcal{YM}(A^g)=\mathcal{YM}(A)$. A critical point $A\in\mathfrak U_G\left(M^4\right)$ of $\mathcal{YM}$ is called a Yang-Mills connection and it is well-known that it satisfies the Euler-Lagrange equation 
  \begin{equation}
  \label{main}
      \diff_A ^* F_A =0.
  \end{equation} 
  Moreover, $F_A$ satisfies also the Bianchi equation 
  \begin{equation}
  \label{bianchi}
      \diff_A  F_A =0.
  \end{equation} 
  %\com{Références sur ce qui est connu : gauge extraction, smoothness, removability of singularities...}
The key ingredient when studying connection with $\mathrm{L}^2$-bounded curvature is the the gauge extraction of Uhlenbeck.
\begin{theorem}[theorem 1.3 \cite{UhlenbeckKarenK1982CwLb}, theorem IV.1 \cite{rivière2015variations} and theorem \ref{extract} of the appendix] There exists $\epsilon_G >0$ and $C_G>0$ such that for all $A\in\mathfrak U_G(\mathrm{B}^4)$ satisfying
$$ \int_{\mathrm{B}^4} \vert F_A\vert^2\, \diff x \leq \epsilon_G, $$
there exists $g\in \mathrm{W}^{1,(4,\infty)}(\mathrm{B}^4,G)$ such that 
$$ \Vert A^g\Vert_{\mathrm{W}^{1,2}} \leq C_G \Vert F_A\Vert_{\mathrm{L}^2} \text{ and } \diff^* A^g=0.$$ \label{uhlenbeckgauge}
\end{theorem}
Here, the theorem is state on $\mathrm{B}^4$, the unit ball of $\R^4$, equipped with the flat metric for sake of clarity but it is also valid on any on ball of $(M^4,h)$. In this particular gauge, the Yang-Mills equation \eqref{main} becomes elliptic and then one can prove the following $\epsilon$-regularity result
\begin{theorem}[theorem VII.1 \cite{rivière2015variations}]
	\label{Ereg} There exists $\epsilon_G >0$ and $C_{G,l}>0$ for every $l\in \N$ such that for all $A\in\mathfrak U_G(\mathrm{B}^4)$ satisfying the Yang-Mills equation \eqref{main} and 
$$ \int_{\mathrm{B}^4} \vert F_A\vert^2\, \diff x \leq \epsilon_G, $$
there exists $g\in \mathrm{W}^{1,(4,\infty)}(\mathrm{B}^4,G)$ such that 
$$ \Vert \nabla^l A^g\Vert_{\mathrm{L}^\infty(\mathrm{B}_{1/2}(0))} \leq C_{G,l} \left( \int_{\mathrm{B}^4} \vert F_A\vert^2\, \diff x\right)^\frac{1}{2}.$$
\end{theorem}

One consequence of the $\epsilon$-regularity result is that sequence of Yang-Mills a connection must satisfies a concentration-compactness phenomena. More precisely, locally either the energy doesn't concentrate and then, up to gauge, the sequence of connections strongly converges to a limit connection, or the energy concentrate and we must rescale the sequence of connections around the concentration point to get a non-trivial weak limit, i.e a Yang-Mills connection over $\R^4$. It is possible that this rescaled sequence of connections is also subject to concentration phenomena. In this case we must also rescale it around its concentration points. Anyway this process stops, since at each step we extract at least one non trivial connection over $\R^4$ whose energy is bounded from below by $\epsilon_G$. This kind of convergence is known as bubble tree convergence, it is characteristic of  conformal invariant problems such as harmonic maps, CMC-surfaces or Willmore surfaces, see for instance \cite{Parker, BrezisCoron,LR14, BernardRiviere, RivConf}. We give here a full statement of the bubble tree convergence, or energy quantization, for Yang-Mills with the notion of convergence developed in appendix \ref{weak}. In full generality it is mainly due to  Rivière \cite{riviere2002interpolation} inspired by Rivière-Lin's work on harmonic maps \cite{RiviereLin}.
\begin{theorem}[theorem VII.3 \cite{rivière2015variations}]
\label{bubbletree}
    Let $ (M^4,h)$ be a closed four-dimensional Riemannian manifold and  $A_k\in\mathfrak {U}_G(M^4)$ a sequence of Yang-Mills connection with uniformly bounded energy. Then we have, up to a sub-sequence,
    \begin{enumerate}
        \item There exists finitely many points $\{p_1,\dots, p_N\}$ and a Yang-Mills connection $A_\infty\in \mathfrak U_G(M^4)$ such that $A_k$ converges to $A_\infty$ in $\mathfrak C^\infty_{G,\mathrm{loc}}(M\setminus \{p_1,\dots, p_N\})$.
        \item For each $i\in \{1, \dots, N\}$, there exists $N_i\in \N$ sequences of points $(p_k^{i,j})_k$ converging to $p_i$,  $N_i$ sequences of scalars $(\lambda_k^{i,j})_k$ converging to zero, $N_i$ non-trivial Yang-Mills connections, called bubbles,  $A_\infty^{i,j} \in\mathfrak{U}_G(S^4)$ such that,
        $$(\phi_{k}^{i,j})^*(A_k) \rightarrow \widehat{A}_\infty^{i,j}=\pi_*A_\infty^{i,j}  \text{ into } \mathfrak C^\infty_{G,\mathrm{loc}}(\R^4 \setminus \{\text{ finitely many points}\}),$$
        where $\phi_{k}^{i,j} (x)= p_k^{i,j} +\lambda _k^{i,j} x$ in local coordinates and $\pi$ is the stereographic projection.
        \item Moreover there is no loss of energy, i.e
        $$ \lim_{k\rightarrow +\infty}\int_{M^4} \vert F_{A_k}\vert^2_h\, \mathrm{vol}_h  = \int_{M^4} \vert F_{A_\infty}\vert^2_h\, \mathrm{vol}_h +\sum_{i=1}^N \sum_{j=1}^{N_i} \int_{S^4} \vert F_{A_\infty^{i,j}}\vert^2\, \mathrm{vol}. $$
    \end{enumerate}
\end{theorem}
One natural question once the quantization of the energy is achieved is to understand how are linked the indices of the sequence of connections and its bubble-tree limits, see section 2 for a definition of the index of a Yang-Mills connection. The Morse index stability for sequences of Yang-Mills connections is of particular interest, motivated by constructing Yang-Mills connections in dimension 4 by min-max procedures. Additionally, the theory of Yang-Mills connections in dimension 4 has many similarities with the theory of harmonic maps in dimension 2. In a recent work  \cite{daliorivieregianocca2022morse}, F. Da Lio, M. Gianocca, and T. Rivière developed a new method to show upper semi-continuity results in geometric analysis, which they applied to conformally invariant Lagrangians in dimension 2 (that include harmonic maps). Their theory was then extended to the the setting of Willmore immersions \cite{michelat2023morseWillmore} and biharmonic maps \cite{michelat2023morseBiharmonic}. If the energy quantitation implies (under suitable assumptions) the lower semi-continuity of the Morse index in a general framework, upper semi-continuity results are typically much more involved%: Da Lio-Gianocca-Rivière showed that the improved energy quantisation implies the upper semi-continuity of the extended Morse index
.  Compared to the lower semi-continuity Morse index estimates, precise pointwise estimates of the sequence of solutions in regions of loss of compactness, the so-called \textit{neck regions}, see below for a definition. Here is the main achievement of the paper
\begin{theorem}
    Let $ (M^4,h)$ be a closed four-dimensional Riemannian manifold and $A_k\in\mathfrak{U}_G(M^4)$ a sequence of Yang-Mills connections with uniformly bounded energy. Let $A_\infty\in \mathfrak U_G(M^4)$ and $A_\infty^{i,j} \in\mathfrak U_G(S^4) $ be its bubble-tree limit in the sense of theorem \ref{bubbletree}. Then, For $k$ large enough, we have
    \begin{enumerate}
        \item  \[ \mathrm{ind}_\mathcal{YM}(A_k) \geqslant \mathrm{ind}_\mathcal{YM}(A_{\infty})+ \sum_{i=1}^N\sum_{j=1}^{N_i} \mathrm{ind}_\mathcal{YM}(A_{\infty}^{i,j}).\]
        \item    \[ \varsigma(A_k) \leqslant \varsigma(A_{\infty})+ \sum_{i=1}^N \sum_{j=1}^{N_i} \varsigma(A_{\infty}^{i,j}),\] where $\varsigma (A)=\ind_{\mathcal{YM}}(A) + \dim \left(\ker Q_{A} \cap \ker \diff_A^*\right)$ is the extended signature of the connection and $Q_A$ is the quadratic form associated to the second variation.
    \end{enumerate}
\end{theorem}

Thanks to the invariance by dilation of the problem, we can decompose it in a finite number of small balls with at most one concentration point. Moreover since concentration phenomena are local, the metric does not play much role, hence we  make proof only for $M^4=\mathrm{B}^4$. Hence we are reduced to study the following situation where 
\begin{enumerate}
        \item There exists $g_k\in \mathrm{W}^{1,2}(\mathrm{B}^4, G)$ such that $A_k^{g_k}$ converges to $A_\infty$ into  $C^\infty_{\mathrm{loc}}(\mathrm{B}^4\setminus \{0\})$.
        \item There exists a sequences of scalars $(\lambda_k)_k$ converging to zero and a non-trivial Yang-Mills connection $\tilde A_\infty \in\mathfrak{U}_G(S^4)$ such that,
        $$(\phi_{k})^*(A_k) \rightarrow \widehat{A}_\infty=\pi_* \tilde A_\infty  \text{ into } C_{\mathrm{loc}}^\infty(\R^4),$$
        where $\phi_{k} (x)= \lambda _k x$ in local coordinates and $\pi$ is the stereographic projection.
        \item Moreover 
        $$ \lim_{k\rightarrow +\infty}\int_{\mathrm{B}^4} \vert F_{A_k}\vert^2\, \mathrm{vol}_h  = \int_{\mathrm{B}^4} \vert F_{A_\infty}\vert^2\, \mathrm{vol}_h +\int_{S^4} \vert F_{\tilde A_\infty} \vert^2\, \mathrm{vol}. $$
    \end{enumerate}
As shown by the previous decomposition, we understand well the behavior of $A_k$ far from $0$ and at scale $\lambda_k$. The crucial point consists in understanding the behavior of $A_k$ in the intermediate region $\mathrm{B}_{\eta}(0) \setminus \mathrm{B}_{\lambda_k/\eta}(0)$ for $\eta$ small, known as \textit{neck region}.\\

\textbf{Acknowledgments:} The authors are very thankful to Tristan Rivière for many useful discussions.

  {~}
\section{Second variation : computation and finiteness of the Morse index}
{~}

We start by remind the first variations of the Yang-Mills functional in our framework. 

\begin{proposition} For every connection $A\in\mathfrak U_G(M)$,  for all $a\in \mathrm{W}^{1,2}(M, T^*M\otimes \mathfrak{g})$, 
     \begin{equation}
         \mathrm{D} \mathcal{YM}_{A} (a) = \langle F_{A}, \diff_{A} a
   \rangle. \label{variationpremiere}
     \end{equation}
\end{proposition}

\begin{proof} Let us show that we have, for $t\in \R$, 
	\begin{equation}
		\label{FAa}
		 F_{A+ta}=F_A + t\diff_A a +t^2 a\wedge a \text{ in } \mathrm{L}^2.
	\end{equation}
	Let $\phi \in \mathcal{C}^\infty_c(M, \Lambda^2 M \otimes \mathfrak g)$ then \begin{align*}
\langle F_{A+ta}, \phi \rangle =&\int_M \langle \diff^*\phi, A+ta\rangle_h +\langle \phi, (A+ta)\wedge (A+t a)\rangle_h \mathrm{vol}_h\\
=&\int_M \langle \diff^*\phi, A\rangle_h +\langle \phi, A\wedge A\rangle_h \mathrm{vol}_h+ t\int_M \langle \diff^*\phi,a\rangle_h +\langle \phi, A\wedge a+A\wedge a\rangle_h \mathrm{vol}_h \\
&+ t^2 \int_M \langle \phi, a\wedge a\rangle_h \mathrm{vol}_h\\
=&\langle F_{A}, \phi \rangle + t \langle \phi, da + A\wedge a + a\wedge a\rangle + t^2 \langle \phi, a\wedge a\rangle \\
=&\langle \phi, F_{A} \rangle + t \langle \phi,\diff_Aa \rangle + t^2 \langle \phi, a\wedge a\rangle .
\end{align*}
Hence we have the desired equality in distributional sense. Moreover we remark that all right-hand side term are in $\mathrm{L}^2$, in particular $\diff_A a$ is in $\mathrm{L}^2$ since by the improved Sobolev embedding, see lecture 32 of \cite{Tartar}, $\mathrm{W}^{1,2}\hookrightarrow \mathrm{L}^{4,2}$ and by Holder inequality on Lorentz spaces, see section 1.4 of \cite{Grafakos1}, $\mathrm{L}^{4,2}.\mathrm{L}^{4,\infty}\hookrightarrow \mathrm{L}^{2,2}=\mathrm{L}^2$. The proposition follows directly from \eqref{FAa}.
\end{proof}

\begin{proposition}
   If $A\in \mathfrak U_G(M)$ is a Yang-Mills connection, for all $a,b\in   \mathrm{W}^{1,2}(M, T^*M\otimes \mathfrak{g})$, \begin{equation}
        \mathrm{D}^2 \mathcal{YM}_{A} (a,b) =  \langle \diff_{A} a, \diff_{A} b \rangle + \langle F_{A}, [a, b] \rangle = \langle L_{A} a,b\rangle \label{variationseconde}
    \end{equation}
where
\[ L_{A} a = \diff_{A}^{\ast} \diff_{A} a +\star [\star F_{A}, a]. \]
\end{proposition} 

\begin{proof} From \[ \diff_{A + a} b = \diff_{A} b + [a, b], \] \[ F_{A + a} - F_{A} = \diff_{A} a + a \wedge a, \] and the embedding $\mathrm{W}^{1, 2} \hookrightarrow \mathrm{L}^{4,2}$, we deduce, by substituting in \eqref{variationpremiere} : \begin{align*}
  \mathrm{D} \mathcal{YM}_{A + a} (b) & = \langle F_{A} + \diff_{A}
  a + a \wedge a, \diff_{A} b + [a, b] \rangle\\
  & = \mathrm{D} \mathcal{YM}_{A} (b) + \langle F_{A}, [a, b] \rangle +
  \langle \diff_{A} a, \diff_{A} b \rangle + \gdo{ \| a
  \|^2_{\mathrm{W}^{1, 2}} \| b\|_{\mathrm{W}^{1, 2}}}
\end{align*}
which implies
\[ \mathrm{D}^2 \mathcal{YM}_{A} (a, b) =    \langle \diff_{A} a, \diff_{A} b \rangle + \langle F_{A}, [a, b] \rangle. \]
Furthermore \begin{align*}
    \tr ([a, b] \wedge \star F_{A}) = \tr ([b, a] \wedge \star F_{A}) &=  \tr\left((b\wedge a + a\wedge b)\wedge \star F_{A} \right) \\
    &=  \tr\left(b\wedge (a\wedge \star F_{A})\right)  +\tr\left( a\wedge (b\wedge \star F_{A}) \right) \\
    &= \tr\left(b\wedge (a\wedge \star F_{A})\right)  -\tr\left( (b\wedge \star F_{A}) \wedge a \right)\\
    &=\tr (b \wedge [a ,\star F_{A}])\\
    &    = -\tr (b \wedge [\star F_{A},a])\\ &= \tr (b \wedge \star \star[\star F_{A},a])
\end{align*} and we obtain the following identity \[\langle F_{A}, [a, b] \rangle = \int_M \tr ([a, b] \wedge \star F_{A}) = \int_M  \tr (b \wedge \star \star[\star F_{A},a]) =  \langle b, \star [\star F_{A}, a] \rangle .\] Therefore the second variation has the following expression:
\[ \mathrm{D}^2 \mathcal{YM}_{A} (a, b) = \langle \diff_{A}^{\ast}
   \diff_{A} a + \star [\star F_{A}, a], b \rangle . \]
\end{proof}

\begin{definition}
	\label{defQ} $ Q_{A}$ is the quadratic form defined by
\[ Q_{A} (a) := \mathrm{D}^2 \mathcal{YM}_{A} (a, a) = \langle L_{A} a, a \rangle. \]
\end{definition}

\begin{remark} The functional $\mathcal{YM}$ is gauge invariant, so is its second derivative, namely \[Q_{A^g}( g^{-1} a g)=  Q_{A}( a) \text{ for all } g\in \mathrm{W}^{1,(4,\infty)}. \] This comes from the pointwise equality : 
\begin{equation}
	\begin{split}
\diff_{A^g} \left(g^{-1} a g\right)  &= d(g^{-1}ag)+[g^{-1}dg+g^{-1}Ag,g^{-1}ag]\\
&=\cancel{dg^{-1}\wedge ag}+g^{-1}dag\cancel{-g^{-1}a\wedge dg}+\cancel{g^{-1}dg \wedge g^{-1}ag}\\
&+g^{-1}A\wedge ag+\cancel{g^{-1}a\wedge dg}+g^{-1}a\wedge Ag\\
&=g^{-1}(da+{A,a})g\\
&= g^{-1}\left(\diff_A a\right) g .\label{gaugeinvariancediff}
\end{split}
\end{equation}

\end{remark}

A consequence of the gauge invariance is that $Q_{A}$ has a huge kernel, as shown in the following result :
\begin{proposition}
   \[ \Imag{\diff_{A}} \subset \ker Q_{A}\]
\end{proposition}

\begin{proof}
Let $\phi \in \mathcal{C}^\infty_c (M, \mathfrak{g})$,we introduce $g_t = \exp(t \phi)$ (well-defined and $G$-valued if $| t |$ is small enough). Since $A^{g_t}$ is a Yang-Mills connection, for all $a \in  \mathrm{W}^{1,2}(M, T^*M\otimes \mathfrak{g})$, \[ \mathrm{D} \mathcal{YM}_{A^{g_t}} (a) = 0, \] which implies, by differentiating with respect to $t$ and evaluating at $t=0$ \[ \mathrm{D}^2 \mathcal{YM}_{A} \left(\left( \frac{\diff}{\diff t}A^{g_t}\right)_{|t = 0}  , a \right) = 0. \] Furthermore, since $g_t= \mathrm{id}+ t \phi +o(t)$ and $\diff g_t = t \diff \phi +o(t)$, we get \begin{align*}
    A^{g_t} = g_t ^{-1} A g_t + g_t^{-1}\diff g_t &= (\mathrm{id} -t \phi +o(t))A(\mathrm{id} +t \phi +o(t)) +(\mathrm{id} + o(1)) (t \diff \phi +o(t))\\ &= A + t \left([A,\phi] +\diff \phi\right) +o(t) \\ &= A +t\, \diff_A\phi +o(t). 
\end{align*} Hence $\displaystyle \left(\frac{\diff}{\diff t}A^{g_t}\right)_{|t = 0} = \diff_A \phi$ and \[ \mathrm{D}^2 \mathcal{YM}_{A} (\diff_{A} \phi, a) = 0 \] \textit{i.e.} $\Imag \diff_{A} \subset \ker Q_{A}$. 
\end{proof}

\noindent Recall the index of a quadratic form $q$ is the non-negative integer $\ind q$ defined as \[ \ind q = \sup \left\{ \dim W \mid q_{|W} < 0\right\}.\]

\begin{proposition} Let $\mathcal{L}_{A} a := \left(\diff_{A}^* \diff_{A} + \diff_{A} \diff_{A}^*\right) a +  \star [\star F_{A}, a] $ and $\mathcal{Q}_{A}(a) =   \langle \mathcal{L}_{A} a , a \rangle$. The index and kernel of $Q_A$ can be expressed from those of $\mathcal{Q}_A$ as follows : \begin{align*}
    \ind Q_{A} &= \ind {\mathcal{Q}_{A}}\\ 
    \ker Q_{A} &=  \ker {\mathcal{Q}_{A}} \oplus  \Imag \diff_{A}  
\end{align*}
\end{proposition}

\begin{proof} Since $\mathcal{Q}_A(a) = \Vert \diff_A^* a \Vert^2 + Q_A(a)$, we have $\mathcal{Q}_A \geqslant Q_A$ and then
\begin{equation} \label{inegaliteindice1}
    \ind \mathcal{Q}_A \leqslant \ind Q_A.
\end{equation} Moreover, since $Q_{A}$ and $\mathcal{Q}_{A}$ coincide on $\ker \diff_{A}^{\ast} = \left( \Imag \diff_A\right)^\perp$, we also have \begin{equation}\label{inegaliteindice2}
    \ind {Q_{A}}_{|\ker \diff_{A}^*} \leqslant \ind \mathcal{Q}_A.
\end{equation} Furthermore, if $W$ is a finite-dimensional subspace on which $Q_A$ is negative-definite, with a $Q_A$-orthogonal base being $(\phi_1,\dots,\phi_n)$, then defining $(\psi_1,\dots,\psi_n)$ their orthogonal projections on $\ker \diff_A^*$, there exist $\eta_1,\dots,\eta_n$ such that for all $i\in\llbracket 1,n\rrbracket$, $\phi_i = \psi_i + \diff_A \eta_i$. Since $\Imag\diff_A \subset \ker Q_A$, the family $(\psi_1,\dots,\psi_n)$ is still $Q_A$-orthogonal (thus linearly independent) and $Q_A$ is negative-definite on their linear span (which is a subset of $\ker \diff_A^*$). We conclude that the following inequality holds : \begin{equation}\label{inegaliteindice3}
    \ind Q_A \leqslant \ind {Q_A}_{|\ker \diff_{A}^*} .
\end{equation} We obtain the first expected equality by combining \eqref{inegaliteindice1},\eqref{inegaliteindice2} and \eqref{inegaliteindice3}.

\noindent For the second point, we already have $\ker {Q_{A}}\cap {\ker \diff_{A}^*} = \ker {\mathcal{Q}_{A}}\cap {\ker \diff_{A}^*}$ and so \[\ker {Q_{A}} =\left( \ker {\mathcal{Q}_{A}}\cap {\ker \diff_{A}^*}\right)\oplus \Imag \diff_A.\] The only thing left to check is \[\ker \mathcal{Q}_{A} \subset \ker \diff_{A}^*.\] If $a\in \ker \mathcal{Q}_A$, we can write $a=b+c$ with $b\in \ker \diff_A^*$ and $c\in \Imag \diff_A$. Using the fact that $c$ is in the kernel $Q_A$ and $\diff_A^*c=\diff_A^*a$, writing the bilinear forms associated to $Q_A$ and $\mathcal{Q}_A$  as $B$ and $\mathcal{B}$, we get \[0=\mathcal{B}(a,c) = B(a,c) + \langle \diff_A^*a,\diff_A^*c \rangle = 0 + \Vert \diff_A^*a\Vert^2.  \] Therefore $\diff_A^*a = 0$.
\end{proof}

\begin{propdef} \label{finitudeindice} Let $\varsigma(A)$ be the extended signature $\ind {Q_{A}} + \dim \left( \ker {Q_{A}}\cap {\ker \diff_{A}^*}\right)$ of the quadratic form $Q_A$. Then \[\varsigma(A) = \ind {\mathcal{Q}_{A}} + \dim  \ker {\mathcal{Q}_{A}}\] and $\varsigma(A)$ is finite.
\end{propdef} 

\begin{proof} The equality comes from the previous proposition. We shall then prove the finiteness of $\varsigma(A)$. Let $\mathcal{E}(\lambda)$ be the (possibly trivial) eigenspace associated with the eigenvalue $\lambda$ of the non-negative elliptic operator $\diff_{A}^* \diff_{A} + \diff_{A} \diff_{A}^*$ . For all $\Lambda\in\R$, we denote \[ \mathcal{E}^{\Lambda} = \bigoplus_{\lambda \leqslant \Lambda} \mathcal{E}(\lambda).
\]
For all $\Lambda \in \R$, $\dim \mathcal{E}^\Lambda < +\infty$. We choose $\Lambda > C \Vert F_{A} \Vert_{\mathrm{L}^\infty(M)}$ where $C$ is such that \[\left| \int_M \langle F_{A}, [a, a] \rangle_h \mathrm{vol}_h \right| \leqslant C  \int_M |F_{A}|_h\, |a|^2_h \mathrm{vol}_h.\]

\noindent If $a\in  \left(\mathcal{E}^\Lambda\right)^\perp$, then \begin{align*}
    \mathcal{Q}_{A}(a) &\geqslant \Lambda \int_M  |a|^2_h \mathrm{vol}_h + \int_M \langle F_{A}, a\wedge a \rangle_h \mathrm{vol}_h\\*
    &\geqslant \int_M \left(\Lambda - C|F_{A}|_h\right) |a|^2_h \mathrm{vol}_h.
\end{align*} We conclude that if $W$ a subspace on which $\mathcal{Q}_{A}$ is negative-definite, $W\cap \left(\mathcal{E}^\Lambda\right)^\perp = \{0\}$. Consequently, $\dim W \leqslant \mathrm{codim}\left(\mathcal{E}^\Lambda\right)^\perp = \dim \mathcal{E}^\Lambda <+\infty$ so $\ind \mathcal{Q}_{A} < +\infty$. A similar argument can be applied to prove that $\ker \mathcal{Q}_{A}$ is also finite dimensional.
\end{proof}

\section{Contribution of the necks to the second variation}
\subsection{Positive contribution of the necks}

{~}

In order to study the behavior of the second variation in the neck regions, the method of Da Lio-Gianocca-Rivière requires sharp pointwise estimates, in contrast to the one coming from the $\varepsilon$-regularity result. Application of $\epsilon$-regularity, see theorem \ref{Ereg}, gives:  for all $R>r>0$ and for all Yang-Mills connection $A$ on $\mathrm{B}_{2R}\backslash \mathrm{B}_{r/2}$, if $E = \| F_{A}
\|_{\mathrm{L}^2(\mathrm{B}_{2R}\backslash \mathrm{B}_{r/2})}\leqslant \varepsilon$ for $x\in \mathrm{B}_{R}\backslash \mathrm{B}_{r}$, \begin{equation} \label{epsreg}
    |F_{A} (x)|\leqslant C \frac{E}{|x|^2} .
\end{equation}

We will explain in the following section why such estimates is not enough in the context of Yang-Mills connections. In other frameworks \cite{daliorivieregianocca2022morse,michelat2023morseWillmore,  michelat2023morseBiharmonic}, refined estimated were obtained using improved energy quantization in Lorentz's spaces. Even if $\mathrm{L}^{2,1}$-energy quantization for Yang-Mills connection is known by the authors \cite{memoireGauvrit}, such estimates can be obtained independently: the following result was from proved in \cite{RadeDecayEstimates} in the flat case and the technical details to extend it to the general case were studied in \cite[theorem I.1]{GroissierParkerSharpEstimates}. For the sake of clarity, we will state the results of this section in the Euclidean framework. They remain true on small enough geodesic balls of a Riemaniann manifold.

\begin{theorem}
    There exist $C,\varepsilon>0$, such that, for all $R>r>0$ and for all Yang-Mills connection $A$ on $\mathrm{B}_{2R}\backslash \mathrm{B}_{r/2}$, if $E = \| F_{A}
\|_{\mathrm{L}^2(\mathrm{B}_{2R}\backslash \mathrm{B}_{r/2})}\leqslant \varepsilon$ then  for $x\in \mathrm{B}_{R}\backslash \mathrm{B}_{r}$, \begin{equation}
    |F_{A} (x)| \leqslant C \frac{E}{|x|^2}\left( \left(\frac{|x|}{R}\right)^2 + \left(\frac{r}{|x|}\right)^2 \right)
\end{equation} \label{sharpestimateneck}
\end{theorem} Let $R>r>0$. In the following, we will write $\omega_{R,r }(x):= \frac{1}{| x |^2} \left( \left(
  \frac{| x |}{R} \right)^{ 2} + \left( \frac{r}{| x |} \right)^{ 2}   \right)$. before to extend this estimate to the connection itself, we give a lemma that allow to cut-off a $\mathrm{W}^{1,2}$-connection in the neck region.
  
  \begin{lemma} \label{cutoffcourbure}
  	Let $A\in \mathcal  A_G^{1,2}(\mathrm{B}_R)$ with $0<R<1$, there exists $C>0$, such that for  $0<r<R$ there exists  $\chi : \mathrm{B}_R\rightarrow [0,1]$ such that $\chi \equiv 1$ on $\mathrm{B}_{R}\setminus \mathrm{B}_r$ and $\mathrm{supp} \chi \subset \mathrm{B}_R \setminus \mathrm{B}_{r/2}$, such that 
  	$$\Vert F_{\chi A}\Vert_{\mathrm{L}^2(\mathrm{B}_R)} \leq  \Vert F_A \Vert_{\mathrm{L}^2(\mathrm{B}_R\setminus \mathrm{B}_{r/2})}  + C \Vert A\Vert_{\mathrm{W}^{1,2}(\mathrm{B}_r\setminus \mathrm{B}_{r/2})}$$ 
  \end{lemma}
  
  \begin{proof}
  	Let $\chi$ a cut-off as in the theorem, with $\Vert d\chi\Vert_\infty \leq 4/r$  , then
  	\begin{align*}
  		F_{\chi A}&=d\chi A +\chi dA + \chi^2 A\wedge A\\
  		&= \chi^2 F_A + (\chi-\chi^2)dA+ d\chi A
  	\end{align*}
  	Then we have 
  	$$\Vert F_{\chi A}\Vert_{\mathrm{L}^2(\mathrm{B}_R)} \leq  \Vert F_A \Vert_{\mathrm{L}^2(\mathrm{B}_R\setminus \mathrm{B}_{r/2})}+\Vert dA\Vert_{\mathrm{L}^2(\mathrm{B}_r\setminus \mathrm{B}_{r/2})} +\frac{4}{r}\Vert A\Vert_{\mathrm{L}^2(\mathrm{B}_r\setminus \mathrm{B}_{r/2})}$$
  	But, by Cauchy-Schwartz,
  	$$\Vert A\Vert_{\mathrm{L}^2(\mathrm{B}_r\setminus \mathrm{B}_{r/2})} \leq  C r \Vert A\Vert_{\mathrm{L}^4(\mathrm{B}_r\setminus \mathrm{B}_{r/2})}$$
  	which achieves the proof thanks Sobolev's embedding.
  \end{proof}

  From the $\varepsilon$-regularity estimate and theorem \ref{sharpestimateneck}, we deduce the following result:
  \begin{corollary} \label{neckC0estimates} There exists $C,\varepsilon$, such that, for all $R>r>0$ and for all connection $A\in \mathfrak U_G(\mathrm{B}_{4R})$ satisfying the Yang-Mills equation on $\mathrm{B}_{2R}\backslash \mathrm{B}_{r/2}$, if $E = \| F_{A}
\|_{\mathrm{L}^2(\mathrm{B}_{4R}\backslash \mathrm{B}_{r/8})}\leqslant \varepsilon$, there exists a gauge $g\in \mathrm{W}^{1,(4,\infty)}(\mathrm{B}_{R}\setminus \mathrm{B}_{r})$ in which, for all $x\in \mathrm{B}_{R}\backslash \mathrm{B}_{r}$, \begin{equation}
    |F_A (x)| + |A^g(x)|^2  \leqslant C E\omega_{R,r }(x).
\end{equation} 
\end{corollary}

\begin{proof} According to lemma \ref{globalg}, there exists $\sigma \in \mathrm{W}^{1,(4,\infty)}(\mathrm{B}_{2R},G)$ such that $A^{\sigma}$ is $\mathrm{W}^{1,2}$ on $\mathrm{B}_{2R}$. Moreover, if $\varepsilon$ is small enough, then we can assume $\|A^{\sigma}
	\|_{\mathrm{W}^{1,2}(\mathrm{B}_{r/2}\backslash \mathrm{B}_{r/4})}\leqslant C\varepsilon$. We can use a cutoff function $\chi$ as in lemma \ref{cutoffcourbure} changing $r$ by $r/2$, we get
	\begin{align*}
	\| F_{\chi A^{\sigma}}
		\|_{\mathrm{L}^2(\mathrm{B}_{2R})}&\leqslant \| F_{ A^{\sigma}}
		\|_{\mathrm{L}^2(\mathrm{B}_{2R}\backslash \mathrm{B}_{ r/4})} +C\|A^{\sigma}
		\|_{\mathrm{W}^{1,2}(\mathrm{B}_{r/2}\backslash \mathrm{B}_{r/4})} \\
		& \leqslant C \varepsilon
	\end{align*} Having $\varepsilon$ small enough, we can therefore assume that $\| F_{\chi A^{\sigma}}
	\|_{\mathrm{L}^2(\mathrm{B}_{2R})}$ is less than the $\varepsilon_G$ of theorem \ref{uhlenbeckgauge}. We deduce the existence of $\tilde{g}\in  \mathrm{W}^{1,(4,\infty)}(\mathrm{B}_{2R},G)$ which is a Coulomb gauge for $A^{\sigma}$ on $\mathrm{B}_{2R}\setminus \mathrm{B}_{r/2}$, meaning that $g:=\sigma \tilde{g}$ is a Coulomb gauge for $A$. Since for all $x\in \mathrm{B}_{R}\backslash \mathrm{B}_{r}$, $\mathrm{B}_{\rho/2}(x) \subset \mathrm{B}_{2\rho} \backslash \mathrm{B}_{\rho} \subset  \mathrm{B}_{2R}\backslash \mathrm{B}_{r/2}$ where $\rho = 2|x|/3$, the $\varepsilon$-regularity estimates from theorem \ref{epsreg} give
\[ |x| |A^g(x)| \leqslant C \frac{\rho}{2}  \|A^g  \|_{\mathrm{L}^\infty (\mathrm{B}_{\rho/4}(x))} \leqslant C  \| F_A \|_{\mathrm{L}^2 (\mathrm{B}_{\rho/2}(x))}\leqslant C \| F_{A} \|_{\mathrm{L}^2 (\mathrm{B}_{2\rho}\backslash\mathrm{B}_\rho)}.\] Using the energy estimate of theorem \ref{sharpestimateneck}, we get \[\| F_{A} \|_{\mathrm{L}^2 (\mathrm{B}_{2\rho}\backslash\mathrm{B}_\rho)} \leqslant C E |\mathrm{B}_{2\rho}\backslash\mathrm{B}_\rho|^{1/2}  \sup_{y\in\mathrm{B}_{2\rho}\backslash\mathrm{B}_\rho} \omega_{R,r}(y) \leqslant C E |x|^2\omega_{R,r}(x).  \]
Finally, since $\omega_{R,r}(x)\leqslant 2/|x|^2$, \[ |A^g(x)|^2 \leqslant \frac{C}{|x|^2} \| F_{A} \|_{\mathrm{L}^2 (\mathrm{B}_{2\rho}\backslash\mathrm{B}_\rho)}^2\leqslant C E |x|^2\omega_{R,r}(x) \times \omega_{R,r}(x) \leqslant C E \omega_{R,r}(x) .\]
\end{proof}
  
\noindent Furthermore, we have the following inequalities: \begin{proposition} \label{GaffneyPoincareNeck}
    There exists $C>0$ such that for all $R>r>0$, for all $\mathrm{W}^{1,2}_0(\mathrm{B}_{R}\backslash \mathrm{B}_{r},\mathfrak{g})$-valued $1$-form $a$,  \begin{equation}
     \int_{\mathrm{B}_{R}\backslash \mathrm{B}_{r}} |a|^2 \omega_{R,r } \diff x  \leqslant C \int_{\mathrm{B}_{R}\backslash \mathrm{B}_{r}} |\diff a|^2 + |\diff^* a|^2 \diff x. 
    \end{equation}
\end{proposition}
\begin{proof}
    It suffices to combine the inequalities from proposition \ref{poincareneck} and  lemma \ref{gaffney}.
\end{proof}

\begin{theorem} \label{pointcrucial1}
    There exist $c_0,\varepsilon>0$, such that, for all $R>r>0$ and for all connection $A\in \mathfrak U_G(\mathrm{B}_{4R})$ satisfying the Yang-Mills equation on $\mathrm{B}_{2R}\backslash \mathrm{B}_{r/2}$ if $E = \| F_{A}
\|_{\mathrm{L}^2(\mathrm{B}_{4R}\backslash \mathrm{B}_{r/8})}\leqslant \varepsilon$ then for all $\mathrm{W}^{1,2}_0(\mathrm{B}_{R}\backslash \mathrm{B}_{r},\mathfrak{g})$-valued $1$- form $a$, \begin{equation}
   \mathcal{Q}_A(a):= \int_{\mathrm{B}_{R}\backslash \mathrm{B}_{r}} \left( |\diff_A a|^2  + |\diff_A^* a|^2 + \langle F_A, [a, a]\rangle \right) \diff x \geqslant c_0  \int_{\mathrm{B}_{R}\backslash \mathrm{B}_{r}} |a|^2 \omega_{R,r } \diff x.
\end{equation} 
\end{theorem}

\begin{proof} Using the fact that, for any gauge $g$, $\diff_{A^g} a^g=\diff a^g +[A^g,a^g]$ and $\diff_{A^g}^* a^g=\diff^* a^g -\star[A^g,\star a^g]$, where $a^g:=g^{-1} a g$, we get :\[\int_{\mathrm{B}_{R}\backslash \mathrm{B}_{r}}  |\diff a^g|^2 \diff x \leqslant \int_{\mathrm{B}_{R}\backslash \mathrm{B}_{r}}  2\left(|\diff_{A^g} a^g|^2 + |[A^g,a^g]|^2 \right) \diff x\] and according to \eqref{gaugeinvariancediff}, \[\int_{\mathrm{B}_{R}\backslash \mathrm{B}_{r}}  |\diff a^g|^2 \diff x\leqslant \int_{\mathrm{B}_{R}\backslash \mathrm{B}_{r}}  2\left(|\diff_{A} a|^2 + |[A^g,a^g]|^2 \right) \diff x \] and similarly \[\int_{\mathrm{B}_{R}\backslash \mathrm{B}_{r}}  |\diff^* a^g|^2 \diff x \leqslant \int_{\mathrm{B}_{R}\backslash \mathrm{B}_{r}}  2\left(|\diff_{A}^* a|^2+ \left|\star[A^g,\star a^g]\right|^2 \right) \diff x. \] We apply this with the gauge from corollary \ref{neckC0estimates} and combine it with proposition \ref{GaffneyPoincareNeck}: \begin{align*}
    \mathcal{Q}_A(a) &\geqslant \int_{\mathrm{B}_{R}\backslash \mathrm{B}_{r}} \left( \frac{1}{2}\left(|\diff a^g|^2+ |\diff^* a^g|^2\right) -C\left(|A^g|^2 +| F_{A} |\right)|a^g|^2\right) \diff x \\
    &\geqslant \left( \frac{1}{2C} - CE\right) \int_{\mathrm{B}_{R}\backslash \mathrm{B}_{r}} |a^g|^2 \omega_{R,r} \diff x = \left( \frac{1}{2C} - CE\right) \int_{\mathrm{B}_{R}\backslash \mathrm{B}_{r}} |a|^2 \omega_{R,r} \diff x.
\end{align*} We complete the proof by taking $\displaystyle \varepsilon \leqslant \frac{1}{4C^2}$ and $\displaystyle c_0 = \frac{1}{4C}$.
\end{proof}

\subsection{The need for sharp neck estimates}
{~}
The goal of this section is to show that, even in the more favorable setting of Yang-Mills connections, the $\varepsilon$-regularity estimates have to be refined. We will use the following weighted Poincaré inequality:

\begin{proposition} There exists $C>0$ such that for all $R>r>0$ and for all $\mathrm{W}_0^{1,2}(\mathrm{B}_R\backslash\mathrm{B}_r,\mathfrak{g})$-valued 1-form $a$:\begin{equation}
    \int_{\mathrm{B}_R \backslash \mathrm{B}_r} \frac{| a |^2}{| x |^2} \diff x \leqslant C \| \nabla
   a \|_{\mathrm{L}^2 \left( \mathrm{B}_R \backslash \mathrm{B}_r \right)}^2. \label{weightedPoincare}
\end{equation}
\end{proposition}

\begin{proof}
See proposition \ref{hardyneck}.
\end{proof}

As a consequence (with the same proof as theorem \ref{pointcrucial1}), for all $\mathrm{W}_0^{1,2}(\mathrm{B}_R\backslash\mathrm{B}_r,\mathfrak{g})$-valued 1-form $a$, \[\mathcal{Q}_A(a) \geqslant c \int_{\mathrm{B}_R \backslash \mathrm{B}_r} \frac{| a |^2}{| x |^2} \diff x,\] which is a first result of positive contribution of the necks to the second derivative. Wanting to diagonalize this quadratic form, we would need compactness in the embedding coming from \eqref{weightedPoincare} (see lemmas \ref{formequaddiagonalisée} and \ref{diagonalisationdelaformequad}). However the embedding $\mathrm{W}^{1,2}_0(\mathrm{B}_1) \hookrightarrow \mathrm{L}^2(\mathrm{B}_1, \diff x/|x|^2 )$ is not compact : fix a non trivial $a \in \mathrm{W}^{1, 2}(\mathrm{B}_1,T^*\mathrm{B}_1\otimes \mathfrak{g})$ with compact support and introduce, for $\varepsilon \in ]0;1]$, $a_{\varepsilon} = \varphi_{\varepsilon}^{\ast} a$ where $\varphi_{\varepsilon} (x) = x / \varepsilon$. Since $\varphi_\varepsilon$ is conformal in $\mathrm{B}_1$, we  \begin{equation*}
    \int_{\mathrm{B}_\varepsilon} \left(|\diff a_\varepsilon|^2 + |\diff^* a_\varepsilon|^2 \right) \diff x  = \int_{\mathrm{B}_1} \left(|\diff a|^2 + |\diff^* a|^2 \right) \diff x, 
\end{equation*}  \begin{equation*}
    \int_{\mathrm{B}_\varepsilon}| a_{\varepsilon} |^2 \diff x = \int_{\mathrm{B}_\varepsilon} \frac{1}{\varepsilon^2} |
   a |^2 \circ \varphi_{\varepsilon} \diff x = \varepsilon^2   \int_{ \mathrm{B}_1}| a |^2 \diff x 
\end{equation*} and \begin{equation*}
   \int_{\mathrm{B}_\varepsilon} \frac{| a_{\varepsilon} |^2}{| x |^2} \diff x = \int_{\mathrm{B}_\varepsilon} \frac{|
   a |^2 \circ \varphi_{\varepsilon}}{| x / \varepsilon |^2} 
   \frac{\diff x}{\varepsilon^4} = \int_{\mathrm{B}_1} \frac{| a |^2}{| x |^2} \diff x.
\end{equation*} 
In particular, $(a_{\varepsilon})$ is bounded in $\mathrm{W}^{1, 2}$.  Furthermore \begin{align}
    \int_{\mathrm{B}_\varepsilon} \frac{|a_\varepsilon|^2}{|x|^2}\diff x =  \int_{\mathrm{B}_1} \frac{|a|^2}{|x|^2}\diff x .\label{contradictioncompacite}
\end{align} Since $a_{\varepsilon} \rightarrow 0$ \textit{a.e.}, if a sequence $(\varepsilon_n)$ converging to zero was such that
$a_{\varepsilon_n} \rightarrow a_0$ in $\mathrm{L}^2(\mathrm{B}_1, \diff x/|x|^2 )$ then Fatou lemma would imply
\[ \int_{\mathrm{B}_1} \frac{| a_0 |^2}{|x|^2} \diff x=  \int_{\mathrm{B}_1}  \liminf_{n \rightarrow + \infty} |
   a_0 - a_{\varepsilon_n} |^2 \frac{\diff x}{|x|^2} \leqslant \liminf_{n
   \rightarrow + \infty} \int_{\mathrm{B}_1}  | a_0 - a_{\varepsilon_n} |^2 \frac{\diff x}{|x|^2} =
   0 \]
\textit{i.e.} $a_0 = 0$ \textit{a.e.} This is a contradiction with \eqref{contradictioncompacite}. Thus, $(a_{\varepsilon})$ has no converging subsequence in
$\mathrm{L}^2(\mathrm{B}_1, \diff x/|x|^2 )$.

\section{Morse index semi-continuity}
\

This section is dedicated to the study of the Morse index stability for Yang-Mills connections. To simplify notations and clarify the exposition, we consider the case where there is only one bubble. It is easy to deduce the general case by splitting the manifold into neck regions, bubbles and thick part and to treat each neck region independently as in the case of single bubble.  Hence we will only be considering the following special case of theorem \ref{bubbletree}: a sequence of Yang-Mills connections $(A_k)$ on a closed 4-manifold $M$ which converges into $\mathfrak C_{G,\mathrm{loc}}^\infty(M\setminus \{p\})$ to $A_\infty \in \mathfrak{U}_G(M)$ a Yang-Mills connection and there exists $p_k \rightarrow p $ and $\lambda_k\rightarrow 0$ such that $\phi_k^* A_k$ converges into $\mathfrak C_{G,\mathrm{loc}}^\infty(\R^4)$ to $\widehat{A}_\infty \in \mathfrak{U}_G(M)$  where $\phi_k(x)= p_k+\lambda_k x$ in local coordinates and $\widehat{A}_\infty=\pi_* \tilde A_\infty$ where $\tilde A_\infty \in \mathfrak{U}_G(S^4)$ is a Yang-Mills connections and $\pi$ is the stereographic projection.\\

Recall that for all $k\in \N$ and $a\in \mathrm{W}^{1,2}(M,T^*M\otimes \mathfrak{g})$, \begin{equation}
     Q_{A_k} (a):= \mathrm{D}^2 \mathcal{YM}_{A_k} (a,a) =\int_{M} | \diff_{A_k} a |_h^2 + \langle F_{A_k}, [a, a] \rangle_h \mathrm{vol}_h.
\end{equation} 

The Morse index of $\mathcal{YM}$ at $A_k$ is defined by: \[ \mathrm{ind}_\mathcal{YM}(A_k):=  \ind Q_{A_k} =  \sup \{ \dim W \mid Q_{\left. A_k\right| W}  < 0 \}. \]  

A first stability result is the lower semi-continuity one:
\begin{theorem} \label{lowerstability}
  For $k$ large enough, \[\mathrm{ind}_\mathcal{YM}(A_k) \geqslant \mathrm{ind}_\mathcal{YM}(A_{\infty})+ \mathrm{ind}_\mathcal{YM}(\widehat{A}_{\infty}).\]
\end{theorem} 

\begin{lemma} Let $a \in \mathrm{W}^{1, 2} (M, T^{\ast} M \otimes
\mathfrak{g})$ and assume there exists $\eta > 0$  such that $\supp a
\subset M\backslash \mathrm{B}_{\eta} (p)$. Then, up to a subsequence, there exists a sequence $(a_k) \subset \mathrm{W}^{1, 2} (M, T^{\ast} M \otimes \mathfrak{g})$, such that satisfying the following properties :
\begin{itemize}
  \item $\supp a_k \subset M\backslash \mathrm{B}_{\eta/2} (p)$,
  
  \item $a_k \overset{\mathrm{L}^{4, 2}}{\rightarrow} a$,
  
  \item $Q_{A_k} (a_k) \rightarrow Q_{A_{\infty}} (a)$.
\end{itemize} \label{approximationlemma}
\end{lemma}

\begin{proof}[Proof of lemma \ref{approximationlemma}]
Consider a finite open cover $\{U_{\alpha}\}_{\alpha}$ of $M\backslash \mathrm{B}_{\eta} (p)$ such that $U_\alpha \subset M\backslash \mathrm{B}_{\eta/2} (p)$ and such that there exist $g_k^{\alpha},
g_{\infty}^{\alpha} \in \mathrm{W}^{1, (4, \infty)} (U_{\alpha}, G)$ satisfying
\[ A_k^{g_k^{\alpha}} 
   \overset{\mathcal{C}^\infty}{\rightarrow} A_{\infty}^{g_{\infty}^{\alpha}}, g^{\alpha \beta}_k :=
   (g_k^{\alpha})^{- 1} g_k^{\beta} 
   \overset{\mathcal{C}^\infty}{\rightarrow} g^{\alpha \beta}_{\infty}
  := (g_{\infty}^{\alpha})^{- 1} g_{\infty}^{\beta}. \]
According to theorem \ref{controlledglobalgauge}, we can always assume that $(A_k)$ is bounded in $\mathrm{L}^{4,\infty}(M)$, which imply that, for all $\alpha$, $(g_k^{\alpha})$ is bounded in $\mathrm{W}^{1,(4,\infty)}$ and, up to a subsequence, we can assume $g_k^{\alpha} \to g_{\infty}^{\alpha}$ a.e. on $U_\alpha$. Introduce $(\chi_{\alpha})_{\alpha}$ a partition of unity subordinate to the open cover $(U_{\alpha})_{\alpha}$. Define
\[ a_k = \sum_{\alpha} \chi_{\alpha} g_k^{\alpha} a_{\alpha} 
   (g_k^{\alpha})^{- 1} \]
where $a_{\alpha} = (g_{\infty}^{\alpha})^{- 1} a
g_{\infty}^{\alpha}$. On $U_{\beta}$,
\[ (g_k^{\beta})^{- 1} a_k g_k^{\beta} = \sum_{\alpha} \chi_{\alpha}
   g_k^{\beta \alpha} g_{\infty}^{\alpha \beta} a_{ \beta} g_{\infty}^{\beta
   \alpha} g_k^{\alpha \beta}. \] We deduce that, on $U_{\beta}$, \[(g_k^{\beta})^{- 1} a_k g_k^{\beta} 
\overset{\mathrm{W}^{1, 2}}{\rightarrow} \sum_{\alpha} \chi_{\alpha} 
(g_{\infty}^{\beta})^{- 1} a g_{\infty}^{\beta} = (g_{\infty}^{\beta})^{- 1} a
g_{\infty}^{\beta}.\] Using the gauge invariance \eqref{gaugeinvariancediff}, we obtain
\begin{align*}
    Q_{A_k}(a_k)=&\sum_{\alpha,\beta}\int_M \left(\langle \diff_{A_k} (\chi_\alpha a_k) , \diff_{A_k} (\chi_\beta a_k)\rangle_h + \langle F_{A_k},  [\chi_\alpha a_k,\chi_\beta a_k] \rangle_h \right) \mathrm{vol}_h \\
    =& \sum_{\alpha,\beta}\int_M \langle \diff_{A_k^{g_k^\alpha}} (\chi_\alpha (g_k^{\alpha})^{-1} a_{k} g_k^{\alpha}) , \diff_{A_k^{g_k^\alpha}} (\chi_\beta g_k^{\alpha \beta }(g_k^{\beta})^{-1} a_{k} g_k^{\beta} g_k^{\beta \alpha})\rangle_h  \mathrm{vol}_h\\
    &+\int_M \chi_{\beta}\langle F_{A_k^{g_k^{\alpha}}}, \chi_\alpha[ (g_k^{\alpha})^{-1} a_{k} g_k^{\alpha},(g_k^{\alpha})^{-1} a_{k} g_k^{\alpha}] \rangle_h \mathrm{vol}_h.
\end{align*} Hence 
\begin{align*}
    Q_{A_k}(a_k) \underset{k\to+\infty}{\rightarrow}& \sum_{\alpha,\beta}\int_M \langle \diff_{A_\infty^{g_\infty^\alpha}} (\chi_\alpha (g_\infty^{\alpha})^{-1} a g_\infty^{\alpha}) , \diff_{A_\infty^{g_\infty^\alpha}} (\chi_\beta g_\infty^{\alpha \beta }(g_\infty^{\beta})^{-1} a g_\infty^{\beta} g_\infty^{\beta \alpha})\rangle_h  \mathrm{vol}_h\\
    &+\int_M \chi_{\beta}\langle F_{A_\infty^{g_\infty^\alpha}}, \chi_\alpha[ (g_\infty^{\alpha})^{-1} a g_\infty^{\alpha},(g_\infty^{\alpha})^{-1} a g_\infty^{\alpha}] \rangle_h \mathrm{vol}_h\\
    &=\sum_{\alpha,\beta}\int_M \left(\langle \diff_{A_\infty} (\chi_\alpha a_\infty) , \diff_{A_\infty} (\chi_\beta a_\infty)\rangle_h + \langle F_{A_\infty},  [\chi_\alpha a_\infty,\chi_\beta a_\infty] \rangle_h\right)\mathrm{vol}_h\\
    &= \int_{M} | \diff_{A_\infty} a |_h^2 + \langle F_{A_\infty}, [a, a] \rangle_h \mathrm{vol}_h = Q_{A_\infty}(a).
\end{align*}
Since  $| a_k |_h \leqslant | a |_h$, the dominated convergence theorem implies the $\mathrm{L}^{4,2}$ convergence.
    \end{proof}

\begin{proof}[Proof of theorem \ref{lowerstability}]
  
Let $a \in \mathrm{W}^{1,2}(M, T^*M\otimes \mathfrak{g})$, $\widehat{a} \in \mathrm{W}^{1,2}(\R^4, T^*\R^4\otimes \mathfrak{g})$. Consider $\chi \in
\mathcal{C}^\infty_c ([0, 2[, [0, 1])$ such that $\chi_{| [0 ; 1]} = 1$. Introduce for $\eta>0$,
 \[  a_\eta (x) = \left( 1 - \chi \left( \frac{| x - p |}{\eta} \right)  \right) a (x), \]
and 
   \[\widehat{a}_{\eta}(x) =\chi \left( 2 \eta |x| \right)
   \widehat{a} \left( x \right).\] Apply lemma \ref{approximationlemma} to $a_{\eta}$ and to $\widehat{a}_{\eta}$ (replacing $M$ by $S^4$) to get two sequences $(a_{\eta,k})_k$ and $(\widehat{a}_{\eta,k})_k$, such that 
   \begin{align*}
   Q_{A_k} (a_{\eta,k}) &\underset{k\to+\infty}{\rightarrow} Q_{A_\infty} (a_{\eta}),\\
    Q_{\phi_k^*A_k} ( \widehat{a}_{\eta,k}) &\underset{k\to+\infty}{\rightarrow} Q_{\widehat{A}_\infty} (\widehat{a}_{\eta}),
   \end{align*} and for all $k$ large enough $a_{\eta,k}$ and $(\phi_k)_*\widehat{a}_{\eta,k}$ have disjoint support. Thanks to the conformal invariance, we have  \[ Q_{A_k} (a_{\eta,k}+ (\phi_k)_*\widehat{a}_{\eta,k}) \underset{k\to+\infty}{\rightarrow} Q_{A_\infty} (a_{\eta})+ Q_{\widehat{A}_\infty} (\widehat{a}_{\eta}).\]
   Since \[\diff_{A_\infty} a_{\eta} = -\diff \left(\chi\left(\frac{|\cdot - p|}{\eta}\right)\right) \wedge a + \left(1-\chi\left(\frac{|\cdot - p|}{\eta}\right)\right)\diff_{A_\infty}a \]  and using the fact that $t\mapsto t \chi'(t)$ is bounded, we obtain 
   \[|\diff_{A_\infty} a_{\eta}| \leqslant C\frac{|a|}{|x-p|} \mathrm{1}_{B(p,2"\eta)}+ |\diff_{A_\infty} a|.\] Using $\mathrm{L}^{4,\infty} \cdot \mathrm{L}^{4,2} \hookrightarrow \mathrm{L}^2$ and the embedding $\mathrm{W}^{1,2}\hookrightarrow \mathrm{L}^{4,2}$, we deduce that the right-hand side is in $\mathrm{L}^2$ and by the dominated convergence theorem $\diff_{A_\infty} a_{\eta}\underset{\eta \to 0}{\rightarrow} \diff_{A_\infty} a$ in $\mathrm{L}^2$. It is straightforward that $a_{\eta}\underset{\eta \to 0}{\rightarrow} a$ in $\mathrm{L}^4$ so we deduce \[\lim_{\eta\to 0} Q_{A_\infty}(a_\eta)=Q_{A_\infty}(a). \] A similar argument show that \[\lim_{\eta\to 0} Q_{\widehat{A}_\infty}(\widehat{a}_\eta)=Q_{\widehat{A}_\infty}(\widehat{a})\] and therefore \[\lim_{\eta \to 0} \lim_{k\to+\infty} Q_{A_k} (a_{\eta,k}+ (\phi_k)_*\widehat{a}_{\eta,k}) = Q_{A_\infty} (a)+ Q_{\widehat{A}_\infty} (\widehat{a}).\] Now, if $Q_{{A}_{\infty}} (a) < 0$ and
$Q_{\widehat{A}_{\infty}} (\widehat{a}) < 0$, for $\eta$ small enough and $k$ large enough, \[Q_{A_k}(a_{\eta,k}+(\phi_k)_* \widehat{a}_{\eta,k})<0.\] \begin{comment}Let $W\subset \mathrm{W}^{1,2}(M, T^*M\otimes \mathfrak{g})$, $\widehat{W} \subset \mathrm{W}^{1,2}(\R^4, T^*\R^4\otimes \mathfrak{g})$ of finite dimension such that $Q_{A_\infty |W}<0$ and $Q_{\widehat{A}_\infty |\widehat{W}}<0$. Consider $(\phi^1,\dots,\phi^n)$ (resp. $(\psi^1,\dots,\psi^m)$) a basis of $W$ (resp. of $\widehat{W}$).\end{comment}

 We have then proved :  \[\liminf_{k \rightarrow + \infty}
\mathrm{ind}_{\mathcal{YM}} (A_k) \geqslant \mathrm{ind}_{\mathcal{YM}}  (A_{\infty}) + \mathrm{ind}_{\mathcal{YM}} 
(\widehat{A}_{\infty}).\]\end{proof}

The upper-semi-continuity result is more subtle. The proof is similar to the one from \cite[theorem I.1]{daliorivieregianocca2022morse}, involving part of the nullity of $Q_{A_k}$ :  \begin{theorem}\label{semicontinuitesupindex}
   For $k$ large enough \[\varsigma(A_k) \leqslant \varsigma(A_{\infty})+\varsigma(\widehat{A}_{\infty}),\] where $\varsigma (A) = \ind_{\mathcal{YM}}(A) + \dim \left(\ker Q_{A} \cap \ker \diff_A^*\right).$
\end{theorem} 
According to proposition \ref{finitudeindice}, the study of $\varsigma(A)$ can be done using the quadratic form $\mathcal{Q}_A$, which has better properties. The key point of the proof is the representation of the quadratic forms  $\mathcal{Q}_{A_k}$ as $a\mapsto  \left( a, \mathcal{L}_k a \right)_k$ where $\mathcal{L}_k$ is a self-adjoint operator with respect to an inner product $\left(\cdot, \cdot \right)_k$. Due to general index consideration related to Sylvester's law of inertia (appendix \ref{Diagformquad}), we can choose the inner product freely : the index and nullity of $\mathcal{Q}_{A_k}$ are independent of such a choice but can be easily determined as dimensions of the negative eigenspaces of $\mathcal{L}_k$ (see lemma \ref{indiceformequaddiago}). The issue is therefore the choice of a suitable inner product, having in mind the estimates of theorem \ref{pointcrucial1}.

For $k\in\N$ and $\eta>0$ small enough, define, in geodesic coordinates around $p$,  the weight function $\omega_{\eta, k}$ by : \[ \omega_{\eta, k} = \left\{\begin{array}{ll}
   \displaystyle \frac{1}{\eta^2} \left( 1 + \left(\frac{\lambda_k}{\eta^2}\right)^2 \right) & \mathrm{if } | x -p_k | \geqslant \eta,\\ \displaystyle \frac{1}{|x-p_k|^2} \left( \left(\frac{|x-p_k|}{\eta}\right)^2 + \left(\frac{\lambda_k}{\eta|x-p_k|}\right)^2 \right) & \mathrm{if } \lambda_k / \eta \leqslant |
     x -p_k | \leqslant \eta,\\
    \displaystyle \frac{\eta^2}{\lambda_k^2}\left( \frac{(1+(1/\eta)^2)^2}{(1+|x-p_k|^2/\lambda_k^2)^2} + \left(\frac{\lambda_k}{\eta^2}\right)^2 \right) & \mathrm{if } |
     x -p_k | \leqslant \lambda_k / \eta\footnotemark.
   \end{array}\right. \] 
   
   \footnotetext{We could have simply extended $\omega_{\eta,k}$ by a constant (depending on $\eta$ and $k$) inside $\mathrm{B}_{\lambda_k/\eta}$. The reason for the choice is because, after pull-back by the stereographic projection, it will become simply a constant.}
   
   We naturally introduce the inner product $\langle \cdot ,\cdot \rangle _{\omega_{\eta,   k}}$ by \[\| a\|_{\omega_{\eta,
   k}}^2 = \int_{M} |a|_h^2 \omega_{\eta,k} \mathrm{vol}_h\] which defines an operator $\mathcal{L}_{\eta, k}$ by \[ \mathcal{Q}_{A_k} (a) = \langle a, \mathcal{L}_{\eta, k} a \rangle _{\omega_{\eta,   k}} \] explicitly given by \[\mathcal{L}_{\eta, k} a = \omega_{\eta, k}^{-1} \mathcal{L}_{A_k}a = \omega_{\eta, k}^{-1} \left( \Delta_{A_k}  a +
  \star[\star F_{A_k}, a]\right).\]

  Let's now introduce the counterparts of $\omega_{\eta,k}$ and $\mathcal{L}_{\eta,k}$ for the principal limit $A_{\infty}$ and the bubble $\widehat{A}_{\infty}$: define, for $x\in M\backslash \{p\}$, \[\omega_{\eta,\infty}(x)= \lim_{k\to +\infty} {\omega_{\eta,k}} (x)= 1/\eta^2 \] with uniform convergence in $M\backslash\mathrm{B}_\eta(p)$, and, 
  \[\widehat\omega_{\eta,\infty}(y)= \lim_{k\to +\infty} \lambda_k ^2 \phi_k^\ast  \left(\omega_{k,\eta}\right)(y) = \left\{\begin{array}{ll} \displaystyle \frac{1}{\eta^2 |y|^4}
  	& \mathrm{if } |
  	y | \geqslant 1/\eta\\
  	\displaystyle \frac{(1 +
  		\eta^2)^2}{\eta^2}  \frac{1}{ (1 + | y |^2)^2} & \mathrm{if } | y| \leqslant 1/ \eta
  \end{array}\right.\] with uniform convergence in $\mathrm{B}_{1/\eta}$. These correspond to the limits of the weight function $\omega_{\eta,k}$ outside and inside the bubble.
  
The estimates of corollary \ref{neckC0estimates} as well as the $\mathrm{L}^2$-quantization, see theorem \ref{bubbletree}, can be used to show the following global estimates:
\begin{proposition}
  \label{pointcrucial2} There exist $\eta_0, \mu_0>0$ such that for all $\eta<\eta_0$, $x\in M$, the following holds for $k$ large enough :
  \begin{equation}
  | F_{A_k} (x) |_h
    \leqslant \mu_0 \omega_{\eta, k} (x), \label{gradientpoids}
  \end{equation}
  and
  \begin{equation}
    \mathrm{Sp} (\mathcal{L}_{\eta, k}) \subset [- \mu_0 ; + \infty [.
    \label{spectreminoré}
  \end{equation}
\end{proposition}
\begin{proof}
    Introduce  \[ \lambda_0 = \limsup_{ \eta \to 0} \limsup_{k \to +\infty} \left\| \frac{| F_{A_k} |_h}{\omega_{\eta, k}}
     \right\|_{\mathrm{L}^{\infty} (M)}.
  \]  It is enough to show that $\lambda_0 =0$. Indeed, if $\lambda \in\mathrm{Sp}
  (\mathcal{L}_{\eta, k})$,
  $a \in \ker (\mathcal{L}_{\eta, k} - \lambda)$, $\| a \|_{\omega_{\eta, k}}
  = 1$,
  \[ \lambda = Q_{A_k} (a) \geqslant - C\int_{M} | F_{ A_k} |_h | a
     |_h^2 \mathrm{vol}_h \geqslant - C \left\| \frac{| F_{A_k} |_h}{\omega_{\eta, k}}
     \right\|_{\mathrm{L}^{\infty} (M)} \| a \|_{\omega_{\eta, k}}^2\] so \eqref{gradientpoids} and \eqref{spectreminoré} hold for $\eta_0$ small enough and $\mu_0$ large enough.
     Since $A_k\underset{k\to+\infty}{\rightarrow} A_\infty$  in $\mathfrak C^\infty_{G,\mathrm{loc}}(M\backslash\{p\})$, then
     \[   \left\| \frac{| F_{A_k} |_h}{\omega_{\eta, k}}
     \right\|_{\mathrm{L}^{\infty} (M \backslash \mathrm{B}_{\eta} (p_k))} \underset{k\to+\infty}{\rightarrow} \left\| \frac{| F_{A_\infty} |_h}{\omega_{\eta, \infty}}
     \right\|_{\mathrm{L}^{\infty} (M \backslash \mathrm{B}_{\eta} (p))} =\eta^2 \left\| | F_{A_\infty} |_h\right\|_{\mathrm{L}^{\infty} (M \backslash \mathrm{B}_{\eta} (p))}. \]
     Since $\phi_k^*A_k\underset{k\to+\infty}{\rightarrow} \widehat{A}_\infty$ in $\mathfrak C^\infty_{G,\mathrm{loc}}(\R^4)$,  \begin{align*}
         \left\| \frac{|F_{A_k} |^2_h}{\omega_{\eta, k}} \right\|_{\mathrm{L}^{\infty} (\mathrm{B}_{\lambda_k / \eta} (p_k))} &= \left\| \frac{|F_{\phi_k^*A_k} |^2_{\phi_k^*h}}{\lambda_k^2 \phi_k^*\omega_{\eta, k}} \right\|_{\mathrm{L}^{\infty} (\mathrm{B}_{1/ \eta})} \\
         &\underset{k\to+\infty}{\rightarrow} \left\| \frac{|F_{\widehat{A}_\infty} |^2}{\widehat{\omega}_{\eta, \infty}} \right\|_{\mathrm{L}^{\infty} (\mathrm{B}_{1/ \eta})} =\frac{\eta^2}{(1+\eta^2)^2}  \left\| (1+|y|^2)^2|F_{\widehat{A}_\infty} |^2 \right\|_{\mathrm{L}^{\infty} (\mathrm{B}_{1/ \eta})}.
     \end{align*} Since $A_{\infty}$ is a smooth connection on $M$ and $\widehat{A}_\infty$ extends to a smooth connection on $S^4$ (by a point removability result, see, \cite[theorem VI.9]{rivière2015variations} for instance), we deduce that \[\eta^2 \left\| | F_{A_\infty} |_h\right\|_{\mathrm{L}^{\infty} (M \backslash \mathrm{B}_{\eta} (p))} +\frac{\eta^2}{(1+\eta^2)^2}  \left\| (1+|y|^2)^2|F_{\widehat{A}_\infty} |^2 \right\|_{\mathrm{L}^{\infty} (\mathrm{B}_{1/ \eta})} =\gdo{\eta^2}. \] It follows using theorem \ref{sharpestimateneck} and the quantization result (see theorem \ref{bubbletree}) :\begin{align*}
         \lambda_0 &=\limsup_{\eta\to 0} \left(\gdo{\eta^2} +  \limsup_{k \to +\infty} \left\| \frac{| F_{A_k} |_h}{\omega_{\eta, k}}
     \right\|_{\mathrm{L}^{\infty} \left(\mathrm{B}_\eta(p_k)\backslash \mathrm{B}_{\lambda_k/\eta}(p_k)\right)} \right)\\
     &\leqslant C\limsup_{\eta\to 0}   \limsup_{k \to +\infty} \left\|  F_{A_k} 
     \right\|_{\mathrm{L}^{2} (\mathrm{B}_{2\eta}(p_k)\backslash \mathrm{B}_{\lambda_k/2\eta}(p_k))}\\
     &\leqslant 0.
     \end{align*} Hence $\lambda_0=0$, which concludes the proof.
\end{proof}

Define now the operators $\mathcal{L}_{\eta, \infty}$ and $\widehat{\mathcal{L}}_{\eta, \infty}$ by \begin{align*}
    \mathcal{L}_{\eta, \infty} a &= \omega_{\eta, \infty}^{-1} \left( \Delta_{A_\infty}  a +
  \star [\star F_{A_\infty}, a]\right),\\
\widehat{\mathcal{L}}_{\eta, \infty} \widehat{a} &= \widehat{\omega}_{\eta, \infty}^{-1} \left( \Delta_{\widehat{A}_\infty}  \widehat{a} +
   \star[ \star F_{\widehat{A}_\infty}, \widehat{a}]\right),
\end{align*} such that \begin{align*}
    \mathcal{Q}_{A_{\infty}} (a) &= \langle a, \mathcal{L}_{\eta, \infty} a \rangle_{\omega_{\eta,
   \infty}},\\
   \mathcal{Q}_{\widehat{A}_{\infty}} (\widehat{a}) &= \langle \widehat{a}, \widehat{\mathcal{L}}_{\eta, \infty} \widehat{a} \rangle_{\widehat{\omega}_{\eta,
   \infty}}
\end{align*} for $a\in \mathrm{W}^{1,2}(M, T^*M\otimes \mathfrak{g})$ and $\widehat{a} \in \mathrm{W}^{1,2}(\R^4, T^*\R^4\otimes \mathfrak{g})$. The operators $\mathcal{L}_{\eta,k}$ (resp. $\mathcal{L}_{\eta,\infty}$) can be diagonalized with respect to the inner products involving the weights $\omega_{\eta,k}$ (resp. $\omega_{\eta,\infty}$) according to lemmas \ref{indiceformequaddiago} and \ref{diagonalisationdelaformequad}. If $\pi : S^4\to \R^4$ is the stereographic projection and $\tilde{A}_\infty\in \mathfrak{U}_G(S^4)$ is such that $\pi^* \widehat{A}_\infty = \tilde{A}_\infty$,  \[ Q_{\tilde{A}_\infty} (\pi^* \widehat{a}) = Q_{\widehat{A}_\infty} (\widehat{a})\] and using lemma \ref{projectionstereosobo}, diagonalizing $Q_{\widehat{A}_\infty}$ (and $\widehat{\mathcal{L}}_{\eta, \infty}$) with respect to the weight $\widehat{\omega}_{\eta,\infty}$ can be reduce to diagonalizing $Q_{\tilde{A}_\infty}$ with respect to \[\pi^*\widehat\omega_{\eta,\infty}(x):= \left\{\begin{array}{ll} \displaystyle \frac{1}{\eta^2} \left(1+|\pi(x)|^{-2}\right)^2
     & \text{ in } 
    \pi^{-1}(\R^4\setminus\mathrm{B}_{1/ \eta})\\
   \displaystyle \frac{(1 +
     \eta^2)^2}{\eta^2} & \text{ in } \pi^{-1}( \mathrm{B}_{1/ \eta}),
   \end{array}\right.\] for which we can also apply we can apply lemmas \ref{indiceformequaddiago} and \ref{diagonalisationdelaformequad}. This guarantees, as stated before, that studying the index of $A_k$ can done by analyzing the spectrum of $\mathcal{L}_{\eta, k}$. More specifically, the proof of theorem \ref{semicontinuitesupindex} comes down to the proof of the following result: \begin{proposition}  \label{majorationindice1}Consider the following spaces \begin{align*}
    W_{\eta, k}&=\displaystyle \bigoplus_{\lambda \leqslant 0} \ker (\mathcal{L}_{\eta, k} - \lambda),\\
    W_{\eta, \infty} &=\displaystyle \bigoplus_{\lambda \leqslant 0} \ker (\mathcal{L}_{\eta, \infty} - \lambda),\\
    \widehat{W}_{\eta,
\infty} &= \displaystyle \bigoplus_{\lambda \leqslant 0} \ker (\widehat{\mathcal{L}}_{\eta, \infty} - \lambda).
\end{align*} There exists $\eta_0>0$ such that for all $\eta\in ]0;\eta_0[$, for all $k$ large enough: \[\dim W_{\eta, k} \leqslant \dim W_{\eta, \infty}  + \dim \widehat{W}_{\eta,
\infty}.\] \end{proposition} \noindent The proof is based on two preliminary results. \begin{lemma} \label{prelim1}
  If $\eta>0$ is small enough, for all $(a_k)_{k \in \N}$ such that for all
  $k \in \N$, $a_k \in W_{\eta, k}$ and $\| a_k \|_{\omega_{\eta, k}} =
  1$, there exists $a_{\infty} \in \mathrm{W}^{1, 2} (M, T^*M \otimes \mathfrak{g})$ and $\widehat{a}_{\infty}\in \mathrm{W}^{1,2}(\R^4, T^*\R^4\otimes\mathfrak{g})$ such that, up to extraction, $a_k \rightharpoonup
  a_{\infty}$ and $\widehat{a}_k \rightharpoonup
  \widehat{a}_{\infty}$ in $\mathrm{W}^{1, 2}$, where $\widehat{a}_k := \phi_k^*a_k$, and for $\delta > 0$ small enough : \begin{equation}
      \int_{M\backslash \mathrm{B}_{\delta} (p)} (| \diff_{A_{\infty}}
       a_{\infty} |^2_h + | \diff_{A_{\infty}}^{\ast} a_{\infty} |^2_h + |
       a_{\infty} |^2_h) \mathrm{vol}_h = \lim_{k \rightarrow + \infty}
       \int_{M\backslash \mathrm{B}_{\delta} (p)} \left( \left|
       \diff_{A_{_k}} a_k \right|^2_h + | \diff_{A_k}^{\ast} a_k |^2_h + |
       a_k |^2_h \right) \mathrm{vol}_h
  \end{equation} and \begin{equation}
       \int_{\mathrm{B}_{1 / \delta}} (| \diff_{\widehat{A}_{\infty}}
       \widehat{a}_{\infty} |^2 + | \diff_{\widehat{A}_{\infty}}^{\ast}
       \widehat{a}_{\infty} |^2 + | a_{\infty} |^2) \diff x  = \lim_{k
       \rightarrow + \infty} \int_{\mathrm{B}_{1 / \delta}} \left( \left|
       \diff_{\widehat{A}_{_k}} \widehat{a}_k \right|^2 + | \diff_{\widehat{A}_k}^{\ast}
       \widehat{a}_k |^2 + | \widehat{a}_k |^2 \right) \diff x.
  \end{equation}
\end{lemma}

\begin{proof} Since
  $\mathcal{Q}_{A_k} (a_k) = \langle a_k, \mathcal{L}_{\eta, k} a_k \rangle \leqslant 0$, according to the first part of proposition \ref{pointcrucial2} the following inequality holds :
  \[ \int_{M} \left(| \diff_{A_k} a_k |^2_h +| \diff_{A_k}^* a_k |^2_h \right) \mathrm{vol}_h \leqslant - \int_{M} \langle F_{A_k} , a_k \wedge a_k \rangle \mathrm{vol}_h \leqslant C\mu_0 \int_{M} | a_k
     |^2 \omega_{\eta, k} \mathrm{vol}_h = C \mu_0. \]   
The following inequality, whose proof is postponed, implies that $(a_k)$ is bounded in $\mathrm{W}^{1,2}$ and the first part of the lemma.     

\vspace{0.6cm}

\textbf{Claim:} For all $\eta$ small enough, there exists $C>0$ such that for all $k$ large enough and for all $a\in \mathrm{W}^{1, 2} (M, T^*M \otimes \mathfrak{g})$, \[ \| a \|_{\mathrm{W}^{1,2}(M)} \leqslant C\left( \| \diff_{A_k} a \|_{\mathrm{L}^{2}(M)} +  \| \diff^*_{A_k} a \|_{\mathrm{L}^{2}(M)} +  \|  a \|_{\omega_{\eta,k}} \right). \]

\vspace{5pt}     
\noindent From the second part of proposition \ref{pointcrucial2}, it follows $\mathrm{Sp}(\mathcal{L}_{\eta,k})\cap \R_- \subset [-\mu_0,0]$ and  \[ \| \mathcal{L}_{\eta, k} a_k \|^2_{  \omega_{\eta, k}} = \langle a_k, \mathcal{L}_{\eta, k} ^2 a_k \rangle _{  \omega_{\eta, k}}\leqslant \mu_0^2 \| a_k \|^2_{\omega_{\eta, k}} = \mu_0^2. \] This gives
  \[ \Delta_{A_k} a_k + \star[\star F_{A_k}, a_k] = \sqrt{\omega_{\eta, k}} f_k \]
  with $(f_k)$ bounded in $\mathrm{L}^2$. Therefore \[ \frac{1}{\sqrt{\omega_{\eta,k}}} \Delta_{A_k} a_k  = f_k - \star\left[ \frac{1}{\omega_{\eta,k}} \star F_{A_k}, \sqrt{\omega_{\eta,k}} a_k\right]\] hence the left-hand side is bounded in $\mathrm{L}^2$. Since, for all $\delta>0$, $\omega_{k,
  \eta}$ is uniformly bounded in $M \backslash \mathrm{B}_{\delta}
  (p)$ (with upper bound depending only on $\eta$ and $\delta$), $(\Delta_{A_k} a_k)_k$ is bounded in $\mathrm{L}^2_{\mathrm{loc}}(M\backslash\{p\})$.

Consider two finite open covers $(U_{\alpha})_{\alpha}, (V_{\alpha})_{\alpha}$ of $M\backslash \mathrm{B}_{\delta} (p)$ such that, for all $\alpha$, $\overline{V_\alpha} \subset U_\alpha$ and  there exist $g_k^{\alpha},
g_{\infty}^{\alpha} \in \mathrm{W}^{1, (4, \infty)} (U_{\alpha}, G)$ satisfying
\[  A_k^{g_k^{\alpha}} 
   \overset{C^\infty(U_\alpha)}{\rightarrow}
   A_{\infty}^{g_{\infty}^{\alpha}}, \qquad g^{\alpha \beta}_k :=
   (g_k^{\alpha})^{- 1} g_k^{\beta} 
   \overset{\mathcal{C}^\infty(U_\alpha\cap U_\beta)}{\rightarrow} g^{\alpha \beta}_{\infty}
 \]
and, similarly to the proof of lemma \ref{approximationlemma}, we can assume that $g_k^{\alpha} \rightarrow g_{\infty}^{\alpha}$ a.e. for all $\alpha$, which gives $ g^{\alpha \beta}_{\infty}= (g_{\infty}^{\alpha})^{- 1} g_{\infty}^{\beta}$. Introduce $(\chi_{\alpha})_{\alpha}$ a partition of unity subordinate to the open cover $\{U_{\alpha}\}_{\alpha}$. On $U_{\alpha}$, similarly to \eqref{gaugeinvariancediff}, the following holds
\[ \Delta_{A_k^{g_k^{\alpha}}} \left( (g_k^{\alpha})^{- 1} a_k g_k^{\alpha} \right)=  (g_k^{\alpha})^{- 1} \left( \Delta_{A_k} a_k \right) g_k^{\alpha}. \] Using an elliptic regularity argument, we deduce that $\left( (g_k^{\alpha})^{- 1} a_k g_k^{\alpha} \right)$ is bounded in $\mathrm{W}^{2,2}(V_\alpha)$. Thus $(g_k^{\alpha})^{- 1} a_k g_k^{\alpha} \to (g_\infty^{\alpha})^{- 1} a_\infty g_\infty^{\alpha}$ strongly in $\mathrm{W}^{1,2}(V_\alpha)$. As in lemma \ref{approximationlemma}, using the gauge invariance \eqref{gaugeinvariancediff}, we obtain \begin{align*}
     \int_{M\backslash \mathrm{B}_{\delta} (p)} \left|
       \diff_{A_{_k}} a_k \right|^2_h  \mathrm{vol}_h=&\sum_{\alpha,\beta}\int_{M\backslash \mathrm{B}_{\delta} (p)} \langle \diff_{A_k} (\chi_\alpha a_k) , \diff_{A_k} (\chi_\beta a_k)\rangle_h \mathrm{vol}_h \\
    =& \sum_{\alpha,\beta}\int_{M\backslash \mathrm{B}_{\delta} (p)} \langle \diff_{A_k^{g_k^\alpha}} (\chi_\alpha (g_k^{\alpha})^{-1} a_{k} g_k^{\alpha}) , \diff_{A_k^{g_k^\alpha}} (\chi_\beta g_k^{\alpha \beta }(g_k^{\beta})^{-1} a_{k} g_k^{\beta} g_k^{\beta \alpha})\rangle_h  \mathrm{vol}_h\\
    \underset{k\to+\infty}{\rightarrow} & \sum_{\alpha,\beta}\int_{M\backslash \mathrm{B}_{\delta} (p)} \langle \diff_{A_\infty^{g_\infty^\alpha}} (\chi_\alpha (g_\infty^{\alpha})^{-1} a_\infty g_\infty^{\alpha}) , \diff_{A_\infty^{g_\infty^\alpha}} (\chi_\beta g_\infty^{\alpha \beta }(g_\infty^{\beta})^{-1} a g_\infty^{\beta} g_\infty^{\beta \alpha})\rangle_h  \mathrm{vol}_h\\
    &=\sum_{\alpha,\beta}\int_{M\backslash \mathrm{B}_{\delta} (p)} \langle \diff_{A_\infty} (\chi_\alpha a_\infty) , \diff_{A_\infty} (\chi_\beta a_\infty)\rangle_h \mathrm{vol}_h\\
    &= \int_{M\backslash \mathrm{B}_{\delta} (p)} | \diff_{A_\infty} a_\infty |_h^2 \mathrm{vol}_h.
\end{align*}
The two other terms are handled in the same way. A similar construction can be done with the sequence $(\widehat{a}_k)$ to construct $\widehat{a}_\infty$ satisfying the required properties.
\end{proof}

\begin{proof}[Proof of the claim] By way of contradiction, assume there exists there exists a sequence $(a_k)$ such that, up to a subsequence, $\Vert a_k \Vert_{\mathrm{W}^{1, 2}} = 1$ and \begin{equation}
\Vert \diff_{A_k} a_k \Vert_{\mathrm{L}^2} + \Vert \diff^{\ast}_{A_k} a_k
\Vert_{\mathrm{L}^2} + \Vert a_k \Vert_{\omega_{\eta, k}} \rightarrow 0. \label{hypotheseclaim}
	\end{equation} Therefore, we have \begin{equation}\label{cvgL2poid} \sqrt{\omega_{\eta, k}} a_k \overset{\mathrm{L}^2}{\rightarrow} 0.
  \end{equation} Since $(a_k)$ is bounded in $\mathrm{W}^{1,2}$ and, up to a subsequence, $a_k \rightharpoonup a$ in $\mathrm{W}^{1, 2}$. By the previous convergence $a = 0$.
  
 \noindent We are going to use the two following identites which hold for all $1$-form $a$ and gauge $g$: \begin{align}
\label{diffgaugeinv} \diff\left( g^{-1} a g\right) &=
\diff_{A_k ^{g}} \left( g^{-1} a g\right) - [A_k^g, g^{-1} a g] =  g^{-1} \left(\diff_{A_k ^{g}} a \right)g - [A_k^g, g^{-1} a g]  \\
\label{diffstargaugeinv} \diff^*\left( g^{-1} a g\right) &=
\diff^*_{A_k ^{g}} \left( g^{-1} a g\right) +\star \left[A_k^g,\star \left( g^{-1} a g\right)\right] = g^{-1} \left(\diff^*_{A_k ^{g}} a \right)g  +\star \left[A_k^g,\star \left( g^{-1} a g\right)\right],
\end{align} where we used the pointwise gauge invariance \eqref{gaugeinvariancediff}.

\noindent Since $A_k\to A_{\infty}$ in $\mathfrak{C}^{\infty}_{G,\mathrm{loc}}(M\setminus\{p\})$, near each point $x \in M\backslash \overline{\mathrm{B}}_{\eta/2}$, there exists a sequence of gauges $(g_k)$ such that $(A_k^{g_k})$ converges smoothly. We can apply (\ref{diffgaugeinv}-\ref{diffstargaugeinv}) to $a=a_k$ and $g=g_k$ and use \eqref{hypotheseclaim} to deduce that ${g_k}^{-1}a_k {g_k} \rightarrow 0$ strongly in $\mathrm{W}^{1,
  2}$ hence in $\mathrm{L}^{4, 2}$ (using Sobolev embedding). The gauge-invariance of the $\mathrm{L}^{4, 2}$ norm imply $a_k \rightarrow 0$
  strongly $\mathrm{L}^{4, 2} \left( M\backslash
  \mathrm{B}_{\eta/2}(p)\right)$. We deduce   \[ \| a_k \|_{\mathrm{L}^{4, 2} \left( M\setminus \mathrm{B}_{\eta}(p_k) \right)}\underset{k \rightarrow + \infty}{\rightarrow} 0. \] We can apply once again (\ref{diffgaugeinv}-\ref{diffstargaugeinv}), with $a=a_k$ and $g=\mathrm{id}$, using the fact that $(A_k)$ is bounded in $\mathrm{L}^{4,\infty}(M)$ to obtain: \begin{equation}
  	\label{limiteW12horsbulle}  \| \diff a_k \|_{\mathrm{L}^{ 2} \left( M\setminus \mathrm{B}_{\eta}(p_k) \right)} + \| \diff^* a_k \|_{\mathrm{L}^{ 2} \left( M\setminus \mathrm{B}_{\eta}(p_k) \right)}\underset{k \rightarrow + \infty}{\rightarrow} 0.
  \end{equation}
    
 \noindent A similar reasoning can be applied to $\phi_k^*A_k$ and $\phi_k^*a_k$ on $\mathrm{B}_{1/\eta}$ using the scale invariance of $\diff_{A_k} a_k$ and $\diff^*_{A_k} a_k$: we get that \begin{align}
 	\label{limiteW12dansbulle}   \| \diff a_k \|_{\mathrm{L}^{ 2} \left( \mathrm{B}_{\lambda_k/\eta}(p_k) \right)} + \| \diff^* a_k \|_{\mathrm{L}^{ 2} \left( \mathrm{B}_{\lambda_k/\eta}(p_k) \right)} =&  \| \diff \left(\phi_k^*a_k \right)\|_{\mathrm{L}^{ 2} \left( \mathrm{B}_{1/\eta}\right)} + \| \diff^*\left(\phi_k^*a_k \right)\|_{\mathrm{L}^{ 2} \left( \mathrm{B}_{1/\eta}\right)} \notag\\ &\underset{k \rightarrow + \infty}{\rightarrow} 0. \end{align}
 
 \noindent	Lastly, since $\eta > 0$ is small enough, according to corollary \ref{neckC0estimates} and the quantization of energy (see theorem \ref{bubbletree}), for $k$ large enough, there exists a sequence of gauges $(g_k)$ in $\mathrm{B}_{2\eta}(p_k) \backslash \mathrm{B}_{\lambda_k / 2\eta}(p_k)$ such that, in this annulus, $| A_k^{g_k} | \leqslant C \sqrt{\omega_{\eta, k}}$ and therefore, if we apply (\ref{diffgaugeinv}-\ref{diffstargaugeinv}), with $a=a_k$ and $g=g_k$, we obtain
  \[ | \diff ({g_k}^{-1}a_k {g_k}) | \leqslant |  {g_k}^{-1} \left(\diff_{A_k} a_k\right) g_k | + C
     \sqrt{\omega_{\eta, k}} | {g_k}^{-1}a_k {g_k} | =  | \diff_{A_k} a_k | + C
     \sqrt{\omega_{\eta, k}} | a_k | \]
  so according to \eqref{hypotheseclaim} \[\| \diff ({g_k}^{-1}a_k {g_k}) \|_{\mathrm{L}^2 \left( \mathrm{B}_{2 \eta}(p_k)
  \backslash \mathrm{B}_{\lambda_k / 2 \eta}(p_k) \right)} \underset{k
  \rightarrow + \infty}{\rightarrow} 0\] and similarly for $\diff^{\ast}({g_k}^{-1}a_k {g_k})$. Consider a cutoff function $\chi_{k, \eta} \in \mathcal{C}^\infty_c \left(\R^4, [0, 1]
  \right)$ satisfying the following properties \begin{itemize}
  	\item $\supp \chi_{\eta,k} \subset 
  	\mathrm{B}_{2 \eta}(p_k) \backslash \overline{\mathrm{B}}_{\lambda_k / 2 \eta}(p_k)$, 
  	\item  $\chi_{k, \eta} = 1$ on $\mathrm{B}_{\eta}(p_k) \backslash 
  	\mathrm{B}_{\lambda_k / \eta}(p_k)$,
  	\item $| \diff \chi_{k, \eta} (x) |
  	\leqslant C / | x -p_k |$.
  \end{itemize}  Notice that, since  $\diff \chi_{k, \eta} = 0$ in $\mathrm{B}_{\eta}(p_k) \backslash
  \mathrm{B}_{\lambda_k / \eta}(p_k)$, the following pointwise estimate holds: \begin{equation}\label{cutoffneck}| \diff \chi_{k, \eta} | \leqslant
  C \sqrt{\omega_{\eta, k}}.\end{equation}
  Since for all $a \in \mathrm{W}^{1, 2} \left(
  \mathrm{B}_{2 \eta}(p_k) \backslash \mathrm{B}_{\lambda_k / 2 \eta}(p_k)
  \right)$, $\chi_{k, \eta} a \in \mathrm{W}^{1, 2} (\mathbf{R}^4)$, according to the Sobolev embedding $\dot{\mathrm{W}}^{1,2}(\R^4)\hookrightarrow \mathrm{L}^{4,2}(\R^4)$ there exists 
  $C > 0$ such that
  \[ \| \chi_{k, \eta} a \|_{\mathrm{L}^{4, 2} (\mathbf{R}^4)} \leqslant C
     \left( \| \diff (\chi_{k, \eta} a) \|_{\mathrm{L}^2 (\mathbf{R}^4)} + \|
     \diff^{\ast} (\chi_{k, \eta} a) \|_{\mathrm{L}^2 (\mathbf{R}^4)} \right) \]
   From \eqref{cutoffneck} we deduce
  \[ | \diff (\chi_{k, \eta} a) | = | \diff\chi_{k, \eta} \wedge a +
     \chi_{k, \eta} \diff a | \leqslant C \sqrt{\omega_{\eta, k}} | a | + |
     \diff a | \]
  and similarly for $\diff^{\ast} (\chi_{k, \eta} a)$. This imply the following 
  \[ \| a \|_{\mathrm{L}^{4, 2} \left( \mathrm{B}_{\eta}(p_k) \backslash
     \mathrm{B}_{\lambda_k / \eta}(p_k) \right)} \leqslant C_{\eta} \left(
     \| \diff a \|_{\mathrm{L}^2 \left( \mathrm{B}_{2 \eta}(p_k) \backslash
     \mathrm{B}_{\lambda_k / 2 \eta}(p_k)\right)} + \| \diff^{\ast} a
     \|_{\mathrm{L}^2 \left( \mathrm{B}_{2 \eta}(p_k) \backslash
     \mathrm{B}_{\lambda_k / 2 \eta}(p_k)\right)} + \| a \|_{\omega_{\eta,
     k}} \right) \]
  Applying this inequality to $({g_k}^{-1}a_k {g_k})_k$, we get   \[ \| a_k \|_{\mathrm{L}^{4, 2} \left( \mathrm{B}_{\eta}(p_k) \backslash
     \mathrm{B}_{\lambda_k / \eta}(p_k) \right)} = \| {g_k}^{-1}a_k {g_k}
     \|_{\mathrm{L}^{4, 2} \left( \mathrm{B}_{\eta}(p_k) \backslash
     \mathrm{B}_{\lambda_k / \eta}(p_k) \right)} \underset{k \rightarrow +
     \infty}{\rightarrow} 0,  \] and as before, we apply (\ref{diffgaugeinv}-\ref{diffstargaugeinv}) with $a=a_k$ and $g=\mathrm{id}$ and use the boundedness of $(A_k)$ in $\mathrm{L}^{4,\infty}(M)$ to conclude  \begin{equation}
     \label{limiteW12neck}    \| \diff a_k \|_{\mathrm{L}^{ 2}  \left( \mathrm{B}_{\eta}(p_k) \backslash
     \mathrm{B}_{\lambda_k / \eta}(p_k) \right)} + \| \diff^* a_k \|_{\mathrm{L}^{ 2}  \left( \mathrm{B}_{\eta}(p_k) \backslash
     \mathrm{B}_{\lambda_k / \eta}(p_k) \right)}\underset{k \rightarrow + \infty}{\rightarrow} 0.\end{equation}
  All in all, combining \eqref{limiteW12horsbulle}, \eqref{limiteW12dansbulle} and \eqref{limiteW12neck} we have \[\| \diff a_k \|_{\mathrm{L}^{ 2}  \left( M \right)} + \| \diff^* a_k \|_{\mathrm{L}^{ 2}  \left(M \right)}\underset{k \rightarrow + \infty}{\rightarrow} 0 \] meaning $a_k \rightarrow 0$ in $\mathrm{W}^{1, 2} (M)$. This is a contradiction with $\| a_k \|_{\mathrm{W}^{1, 2}} = 1$.
\end{proof}

\begin{lemma} \label{prelim2}
  With the notations of the previous lemma, $(a_{\infty},
  \widehat{a}_{\infty}) \neq (0, 0)$.
\end{lemma}

\begin{proof} We prove this result by contradiction, using theorem \ref{pointcrucial1}. Assume $(a_{\infty},
  \widehat{a}_{\infty}) = (0, 0)$. Therefore, we have \begin{align} 
       \lim_{k \rightarrow + \infty}
      & \int_{M\backslash \mathrm{B}_{\delta} (p)} \left( \left|
       \diff_{A_{_k}} a_k \right|^2_h + | \diff_{A_k}^{\ast} a_k |^2_h + |
       a_k |^2_h \right) \mathrm{vol}_h =0 \label{hypabs1}\\
 \lim_{k
       \rightarrow + \infty} &\int_{\mathrm{B}_{1 / \delta}} \left( \left|
       \diff_{\widehat{A}_{_k}} \widehat{a}_k \right|^2 + | \diff_{\widehat{A}_k}^{\ast}
       \widehat{a}_k |^2 + | \widehat{a}_k |^2 \right) \diff x =0. \label{hypabs2}
  \end{align}
  
  Let's introduce a cut-off function  $\chi \in \mathcal{C}^\infty (\R_+, [0 ; 1])$ with $\chi_{| [0 ;
  1]} = 1$ and $\mathrm{supp} \chi \subset [0 ; 2]$. Define a new sequence $(\check{a}_k)$ by $\check{a}_k = \chi_k a_k$ where
  \[ \chi_k(x) = \chi \left( 2 \frac{| x - p_k |}{\eta} \right) \left(
     1 - \chi \left( \eta \frac{| x - p_k |}{\lambda_k} \right) \right)
     . \] It is straightforward to check that $\mathrm{supp}\, \check{a}_k \subset \mathrm{B}_{\eta} (p_k) \backslash
  \mathrm{B}_{\lambda_k / \eta} (p_k)$ \textit{i.e.} $ \check{a}_k \in  \mathrm{W}^{1, 2}_0
  \left( \mathrm{B}_{\eta} (p_k) \backslash \mathrm{B}_{\lambda_k / \eta} (p_k)
  \right)$ and  $\check{a}_k  = a_k $ on $\mathrm{B}_{\eta / 2} (p_k) \backslash \mathrm{B}_{2
  	\lambda_k / \eta} (p_k)$.
  
  \noindent Therefore, if $k$ is large enough, \begin{eqnarray*}
    \int_M | \diff_{A_k} (a_k - \check{a}_k) |^2 & = & \int_{M\backslash
    \mathrm{B}_{\eta / 2} (p_k)} | \diff_{A_k} (a_k - \check{a}_k) |^2 +
    \int_{\mathrm{B}_{2 \lambda_k \backslash \eta} (p_k)} | \diff_{A_k} (a_k -
    \check{a}_k) |^2
  \end{eqnarray*} Using the formula  
  \[ \diff_{A_k} (a_k - \check{a}_k) = - \diff \chi_k \wedge a_k + (1 -
     \chi_k) \diff_{A_k} a_k \]
  we deduce the following point-wise estimates in $M\backslash \mathrm{B}_{\eta / 2} (p_k)$ :
  \[ | \diff_{A_k} (a_k - \check{a}_k) |^2 \leqslant C \left( \frac{| a_k
     |^2}{\eta^2} + | \diff_{A_k} a_k |^2 \right) \]
 and in $\mathrm{B}_{2 \lambda_k / \eta} (p_k)$ :
  \[ | \diff_{A_k} (a_k - \check{a}_k) |^2 \leqslant | \diff_{A_k} (a_k -
     \check{a}_k) |^2 \leqslant C \left( \frac{| a_k |^2}{(\lambda_k / \eta)^2}
     + | \diff_{A_k} a_k |^2 \right) \]
 which lead to \begin{align*}
      \| \diff_{A_k} (a_k - \check{a}_k) \|_{\mathrm{L}^2 (M)}^2  \leqslant &
    C_{\eta} \left( \| a_k \|_{\mathrm{L}^2 \left( M\backslash \mathrm{B}_{\eta /
    2} (p_k) \right)}^2 + \| \diff_{A_k} a_k \|^2_{\mathrm{L}^2 \left(
    M\backslash \mathrm{B}_{\eta / 2} (p_k) \right)} \right)\\
    & + C_{\eta} \left( \| \widehat{a}_k \|_{\mathrm{L}^2 \left( \mathrm{B}_{2 /
    \eta} \right)}^2 + \| \diff_{\widehat{A}_k} \widehat{a}_k \|_{\mathrm{L}^2 \left(
    \mathrm{B}_{2 / \eta} \right)}^2 \right).
 \end{align*} Combining it with \eqref{hypabs1} and \eqref{hypabs2}, we get\[\| \diff_{A_k} a_k - \diff_{A_k} \check{a}_k
  \|_{\mathrm{L}^2 (M)} \underset{k\to+\infty}{\rightarrow} 0.\] Similarly \[\| \diff_{A_k}^{\ast} a_k
  - \diff_{A_k}^{\ast} \check{a}_k \|_{\mathrm{L}^2 (M)} \underset{k\to+\infty}{\rightarrow} 0.\] Additionally,
  \[ \omega_{\eta, k} | a_k - \check{a}_k |^2 = \omega_{\eta, k} (1 -
     \chi_k)^2 | a_k |^2 \]
  which gives
  \[ \| a_k - \check{a}_k \|_{\omega_{\eta, k}}^2 \leqslant C_{\eta} \left(
     \| a_k \|^2_{\mathrm{L}^2 \left( M\backslash \mathrm{B}_{\eta / 2} (p)
     \right)} + \| \widehat{a}_k \|^2_{\mathrm{L}^2 \left( \mathrm{B}_{2 / \eta}
     \right)} \right) \]
  and proves that \[\| a_k - \check{a}_k \|_{\omega_{\eta, k}} \underset{k\to+\infty}{\rightarrow} 0.\]
  Finally, since $a_k \wedge \check{a}_k = \check{a}_k \wedge a_k = \chi_k a_k
  \wedge a_k$, we have \begin{align*}
      \left| \int_M (\langle F_{A_k}, a_k \wedge a_k \rangle - \langle F_{A_k},
    \check{a}_k \wedge \check{a}_k \rangle) \right| & =  \left| \int_M
    \langle F_{A_k}, (a_k + \check{a}_k) \wedge (a_k - \check{a}_k) \rangle
    \right|\\
    & \leqslant  C \int_M \omega_{\eta, k} | a_k + \check{a}_k | | a_k -
    \check{a}_k |\\
    & \leqslant  C \| a_k + \check{a}_k \|_{\omega_{\eta, k}} \| a_k -
    \check{a}_k \|_{\omega_{\eta, k}}.
  \end{align*} We conclude that $\displaystyle \| \check{a}_k \|_{\omega_{\eta, k}} \underset{k \rightarrow +
     \infty}{\rightarrow} 1$ and $\displaystyle \mathcal{Q}_{A_k} (a_k)
  - \mathcal{Q}_{A_k} (\check{a}_k) \underset{k \rightarrow +
     \infty}{\rightarrow} 0$. However, according to theorem \ref{pointcrucial1}, $\mathcal{Q}_{A_k} (\check{a}_k) \geqslant c_{0} \|
  \check{a}_k \|_{\omega_{\eta, k}}$ so
  \[ \liminf_{k\to+\infty}  \mathcal{Q}_{A_k} (a_k) \geqslant c_{0} > 0, \]
  which contradicts the fact that $a_k \in W_{\eta, k}$.
\end{proof}

\begin{proof}[Proof of proposition \ref{majorationindice1}]
  Let $N \in \N^{\ast}$ such that there exists infinitely many $k \in
  \N$ satisfying $\dim W_{\eta, k} \geqslant N$. For every such $k$, let $(\phi_k^j)_{j \in \llbracket
  1 ; N \rrbracket}$ be an orthonormal family (for $\langle  \cdot, \cdot 
  \rangle_{\omega_{\eta, k}}$) of $W_{\eta, k}$, made up of eigenvectors of $\mathcal{L}_{\eta, k}$. Write $\lambda_k^j$, the corresponding eigenvalues. Using the second part of \ref{pointcrucial2},
 up to extraction, $(\lambda_k^j)_{k \in
  \N}$ converge for all $j$. Define the (non-positive) limit as $\lambda_{\infty}^j$. Denote $(\phi^j_{\infty},
  \widehat{\phi}^j_{\infty})$ the corresponding limits given by lemma \ref{prelim1}, with $(\phi^j_{\infty},
  \widehat{\phi}^j_{\infty}) \neq (0, 0)$ (according to lemma \ref{prelim2}) satisfying \begin{align*}
      \mathcal{L}_{\eta, \infty} \phi^j_{\infty} &= \lambda_{\infty}^j
     \phi^j_{\infty},\\* \widehat{\mathcal{L}}_{\eta, \infty} \widehat{\phi}^j_{\infty} &=
     \lambda_{\infty}^j \widehat{\phi}^j_{\infty}.
  \end{align*}
  In particular,  $(\phi^j_{\infty}, \widehat{\phi}^j_{\infty}) \in W_{\eta, \infty} \times
  \widehat{W}_{\eta, \infty}$. If $N > \dim W_{\eta, \infty} + \dim \widehat{W}_{\eta,
  \infty}$, there exists a linear relation \begin{equation} \label{dependance}
      \sum_{j = 1}^N c_j (\phi^j_{\infty}, \widehat{\phi}^j_{\infty}) = 0, 
  \end{equation}
  where $\displaystyle\sum_{j = 1}^N |c_j|^2 = 1$. If we then define $a_k = \displaystyle\sum_{j = 1}^N c_j
  \phi^j_k$, it satisfies $a_k \in W_{\eta, k}$, $\| a_k \|_{\omega_{\eta, k}} = 1$, and \eqref{dependance} implies that the corresponding $(a_\infty, \widehat{a}_\infty)$ given by lemma \ref{prelim1} both vanish, which contradicts lemma \ref{prelim2}. The following inequality thus holds : $N \leqslant \dim W_{\eta, \infty} + \dim
  \widehat{W}_{\eta, \infty}$.
  
  We have proved there exist only finitely many $k \in
  \N$ such that $\dim W_{\eta, k} > \dim W_{\eta, \infty} + \dim
  \widehat{W}_{\eta, \infty}$, so for $k$ large enough,
  \[ \dim W_{\eta, k} \leqslant \dim W_{\eta, \infty} + \dim \widehat{W}_{\eta,
     \infty} . \] 
\end{proof}

 \newpage
 \appendix
\section{Annexes}
\subsection{weak connections and its notion of convergence}
\label{weak}
In this section we give some details about the notion of weak-connection developed by Petrache and Rivière in \cite{PETRACHE2017469}.\\

We start with the notion of weak-principal bundle introduced by Isobe, see \cite{Isobe1,Isobe2}. We set 
$$\mathcal P_G^{k,p}(M) :=\left\{ \begin{array}{l}
    P=(\{ U_\alpha\}_{\alpha\in I}, \{g^{\alpha\beta}\}_{\alpha,\beta \in I} ) \text{ s.t. } \{ U_\alpha\}_{\alpha\in I} \text{ is a good cover\footnotemark of }M,\\
    g^{\alpha\beta}\in \mathrm{W}^{k,p}(U_\alpha \cap U_\beta,G) \text{ satisfy } g^{\alpha\beta}.g^{\beta\gamma} = g^{\alpha\gamma} \text{ a.e. on } U_\alpha \cap U_\beta \cap U_\gamma .
\end{array}\right\}
$$
\footnotetext{That is to say each $U_\alpha$ and all of its intersections are diffeomorphic to a ball.}

The fact that $\mathcal{C}^\infty(\mathrm{B}^4,G)$ is dense in $\mathrm{W}^{2,2}(\mathrm{B}^4,G)$ leads to the fact that $\mathcal P^{2,2}_G$-bundles can be approximated by smooth bundles, see \cite{Sil}. But as well explain in remark 22 of \cite{Sil}, the isomorphism class of the approximation is \textit{a priori} not unique. Nevertheless, as soon as , the $\mathrm{W}^{2,2}$-bundle $P$ is equipped with a $\mathrm{W}^{1,2}$-connection $A$, see below, the isomorphism class of the approximation associated to the pair $(P,A)$ becomes unique. Indeed, the local trivialization are given the Uhlenbeck's Coulomb gauge of the connection, then, as proved in theorem V.5 of \cite{rivière2015variations}, the change of gauge are $C^0$ and  then approximate by smooth bundle. So let us introduce the space of $\mathrm{W}^{1,2}$-connections. We can then define the space of $\mathrm{W}^{1,2}$-connections on $P=(\{ U_\alpha\}_{\alpha\in I}, \{g^{\alpha\beta}\}_{\alpha,\beta \in I} ) \in \mathcal P^{2,2}(M)$, as 
$$\mathcal A_G^{1,2}(P) :=\left\{ \begin{array}{l}
	A=\{ A_\alpha\}_{\alpha\in I} \text{ s.t. } A_\alpha\in \mathrm{W}^{1,2}(U_\alpha, \Lambda^1 U_\alpha \otimes \mathfrak g) \text{ for all } \alpha \in I \\
	\text{and }A_\beta =(g^{\alpha\beta})^{-1} dg^{\alpha\beta} +(g^{\alpha\beta})^{-1} A_\alpha g^{\alpha\beta}  \text{ on } U_\alpha \cap U_\beta .
\end{array}\right\}
$$
 We will say that two couple of bundle-connection $(P,A),(P',A')$ such that $P,P' \in \mathcal P_G^{2,2}(M)$,  $A\in \mathcal A_G^{1,2}(P)$  and $A'\in \mathcal A_G^{1,2}(P')$ are $\mathrm{W}^{k,l}$\textbf{-equivalent}, we write it $(P',A')=k((P,A))$, if there exists a good cover, which is a refinement of the cover that define $P$ and $P'$, $\{ U_\alpha\}_{\alpha\in I}$ such that $P=(\{ U_\alpha\}_{\alpha\in I}, \{g^{\alpha\beta}\}_{\alpha,\beta \in I} )$ and $P'=(\{ U_\alpha\}_{\alpha\in I}, \{h^{\alpha\beta}\}_{\alpha,\beta \in I} )$, and they exists $k^{\alpha}\in \mathrm{W}^{k,l}(U_\alpha, G)$ such that 
$$h^{\alpha\beta}=k^{\alpha} g^{\alpha\beta}(k^{\beta})^{-1} \text{ a.e. on } U_\alpha \cap U_\beta$$
and
$$A_\alpha=(k^{\alpha})^{-1} dk^{\alpha}  +(k^{\alpha})^{-1}A_{\alpha}'k^{\alpha}  \text{ a.e. on } U_\alpha.$$
We can then consider classes of equivalent weak principal bundles
$$ [(P,A)]_{k,l}=\{(P',A') \text{ s.t. } (P',A')=k((P,A)) \text{ for some } k^{\alpha}\in \mathrm{W}^{k,l}(U_\alpha, G)\} .$$
It is important to note that, thanks to the approximation result of \cite{Sil}, $\mathrm{W}^{2,2}$-equivalent bundles are topologically equivalent. But, because we can build global section\footnote{ It is a consequence of the fact that any map from $\mathrm{W}^{1,3}(S^3,S^3)$ can be extended to a map $\mathrm{W}^{2,p}(\mathrm{B}^4,S^3)$ for $p<2$, see XXX.} into $\mathrm{W}^{2,p}$ for $p<2$, then there is only one class of $\mathrm{W}^{2,p}$-equivalent bundles for $p<2$, namely the trivial class. We denote
$$\hat{\mathcal P}^{2,2}(M)=\{[(P,A)]_{2,2} \text{ s.t. } P \in \mathcal P^{2,2}(M) \text{ with } A \in \mathcal A_G^{1,2}(P)  \} $$
Here are two remarks. First, see \cite{MW},  a $\mathcal{C}^0$-bundle can be approximate by a smooth bundle: there exists  $(g_k^{\alpha \beta})_k \in \mathcal{C}^\infty(U_\alpha\cap U_\beta, G)$ satisfying the cocylce condition such that $g_k^{\alpha \beta} \rightarrow g^{\alpha\beta}$ in  $\mathrm{L}^{\infty}_{\mathrm{loc}} \left( U_\alpha\cap U_\beta\right)$. Secondly, see section 10 chapter 2 \cite{BottTu},  a smooth bundle over $\mathrm{B}^4$ is trivial, i.e we can find some $\rho_{\alpha} \in  C^{\infty}(U_\alpha, G)$ such that $g^{\alpha\beta} = (\rho^{\alpha})^{- 1} \rho^{\beta}$. The following result, lemma \ref{Ulem}, is a weak version of those  remarks, i.e. it proves the existence of $\rho^\alpha$ for a continuous $\mathcal P^{2,2}_G$-bundle with a $\mathrm{L}^\infty\cap \mathrm{W}^{2,2}$ control on the $\rho^\alpha$. The lemma is due to Uhlenbeck in the case of $\mathrm{W}^{2,p}$-bundles with $p>2$, see section 3 of \cite{UhlenbeckKarenK1982CwLb}, we improve it to $\mathrm{W}^{2,2}\cap C^0$-connections. In fact the key point to prove the lemma consists in being able to restrict to gauge in $\tilde G$ where $\tilde G$ is neighborhood of $\mathrm{id}$ in $G$ where  $\exp^{-1}: \tilde G \rightarrow \mathfrak g$ is well defined. This is where the $C^0$ assumption is crucial. We also remind the following useful lemma, see lemma 3.1 \cite{UhlenbeckKarenK1982CwLb}.
\begin{lemma}
	\label{techlemma} Let $G$ be a compact Lie group with an equivariant metric. Then there exists $f_G>0$ such that if $h,g,\rho \in G$ with $\vert \exp^{-1} hg\vert \leq f_G$ and $\vert \exp^{-1} \rho\vert <f_G$, then $h\rho g\in \tilde G$ and 
	$$ \vert \exp^{-1} h\rho g\vert \leq 2(\vert\exp ^{-1} hg\vert +\vert \exp^{-1} \rho \vert ).$$
\end{lemma}

\begin{lemma}
	\label{Ulem}
	Let $\{U_\alpha \}_{\alpha \in I}$ be a good cover of $\mathrm{B}^4$. There exists $\delta_G >0$, depending only on $G$, such that for any pair of co-chains,
	$$g^{\alpha\beta}, h^{\alpha\beta} \in \mathrm{W}^{2,2}\cap C^0(U_\alpha\cap U_\beta,G)$$
	satisfying
	$$g^{\alpha\gamma}g^{\gamma\beta}=g^{\alpha\beta} \text{ and }  h^{\alpha\gamma}h^{\gamma\beta}=h^{\alpha\beta} \text{ on } U_\alpha \cap U_\beta \cap U_\gamma \text{ for all } \alpha, \beta,\gamma \in I$$
	and
	$$m:= \max_{\alpha,\beta}\Vert g^{\beta\alpha}h^{\alpha\beta}-\mathrm{id}\Vert_{\mathrm{L}^\infty(U_\alpha\cap U_\beta)}<\delta_G.$$
	Then there exists a constant $C>0$, depending only on $G$ and $\vert I\vert$, a refinement $V_\alpha \subset U_\alpha$ and $\rho^\alpha \in \mathrm{W}^{2,2}\cap C^0(V_\alpha,G)$ such that 
	$$ h^{\alpha\beta}=(\rho^\alpha)^{-1}g^{\alpha\beta}\rho^\beta \text{ on } V_\alpha \cap V_\beta \text{ for all } \alpha,\beta \in I,$$
	moreover
	$$ \Vert \rho^{\alpha}-\mathrm{id} \Vert_{\mathrm{L}^\infty (V_\alpha)} \leq C m$$
	and 
	$$ \Vert \rho^{\alpha}-\mathrm{id} \Vert_{\mathrm{W}^{2,2} (V_\alpha)} \leq  m  P( (\Vert g^{ab} \Vert_{\mathrm{W}^{2,2} (U_a\cap U_b)})_{a,b\in I},(\Vert h^{ab} \Vert_{\mathrm{W}^{2,2} (U_a\cap U_b)})_{a,b\in I}), $$
	where $P$ is polynomial expression whose coefficients depends only on $G$ and $\{ U_\alpha\}$.
\end{lemma}
As explain above the proof is exactly the same than the one of Uhlenbeck with some more details on the $\mathrm{W}^{2,2}$-control.
\begin{proof}
	The proof is made inductively on the number of element of the covering for which we assume that $I=\{1,\dots, N\}$ and $m<f_G$ where $f_G$ is given by lemma \ref{techlemma}. We set $U_{1,1}=U_1$ and $\rho_1 : U_1 \rightarrow G$ being $\rho_1 \equiv \mathrm{id}$. Then we assume that we are at step $k$, i.e we have $\{U_{\alpha,k}\}_{\alpha \in \{1, \dots ,k\}}$ such that $U_{\alpha, k} \subset U_\alpha$ and 
	$$ \mathrm{B}^4\subset \bigcup_{\alpha \leq k}U_{\alpha, k} \bigcup_{\alpha >k} U_\alpha, $$
	we have $\rho_\alpha \in   \mathrm{W}^{2,2}\cap C^0(U_{\alpha,k},G)$ such that 
	\begin{equation}
		\label{consist}
	 h^{\alpha\beta}=(\rho^\alpha)^{-1}g^{\alpha\beta}\rho^\beta \text{ on } U_{\alpha, k} \cap U_{\beta, k} \text{ for all } \alpha,\beta \in \{1,\dots,k\},
	\end{equation}
	moreover, there exists $C_k>0$ depending only on $k$ and $P_k$ a polynomial expression depending  only on $G$ and $\{U_\alpha\}_{\alpha\leq k}$ such that, for all $\alpha \leq k$,
	$$ \Vert \rho^{\alpha}-\mathrm{id} \Vert_{\mathrm{L}^\infty (V_\alpha)} \leq C_k m$$
	and 
	\begin{equation}
		\label{Pk}
	 \Vert \rho^{\alpha}-\mathrm{id} \Vert_{\mathrm{W}^{2,2} (V_\alpha)} \leq m P_k( (\Vert g^{ab} \Vert_{\mathrm{W}^{2,2} (U_a\cap U_b)})_{a,b \leq k},(\Vert h^{ab} \Vert_{\mathrm{W}^{2,2} (U_a\cap U_b)})_{a,b \leq k}).
	\end{equation}
	Then, thanks to lemma \ref{techlemma}, assuming $\delta_G< f_G/C_k$, for all $\alpha \in \{1,\dots, k\}$, we can define $u_{k+1}: U_{k+1}\cap U_{\alpha,k}\rightarrow \mathfrak g$ by $ u_{k+1} =\exp^{-1}(g^{k+1,\alpha}\rho^\alpha h^{\alpha,k+1})$. It follows by \eqref{consist} that $u_{k+1}$ is consistently define on $U_{k+1} \cap \cup_{\alpha \leq k} U_{\alpha, k}$ and by lemma \ref{techlemma} we have $ \Vert u_{k+1} \Vert_{\mathrm{L}^\infty} \leq 2(C_k +1)m$. Let $\chi_\alpha$ a partition of unity associated to  $\{U_{\alpha, k}\}_{\alpha \leq k} \cup \{ U_\alpha\}_{\alpha >k}, $ we set $\displaystyle \phi_{k+1}=\sum_{\alpha \leq k} \chi_\alpha$,
	$$U_{\alpha,k+1}=U_{\alpha,k} \cap \text{interior}\{x\, \vert \, \phi_{k+1}=1\}\text{ for } \alpha\leq k$$
	and 
	$$U_{k+1,k+1}=U_{k+1}.$$
	We easily check  that
	$$ \mathrm{B}^4\subset \bigcup_{\alpha \leq k+1}U_{\alpha, k+1} \bigcup_{\alpha >k+1} U_\alpha. $$ 
	We then set $\rho^{k+1}= \exp(\phi_{k+1}u_{k+1})$ on $U_{k+1} \cap \cup_{\alpha \leq k} U_{\alpha, k}$ and $\mathrm{id}$ elsewhere. We clearly have 
	$$h^{\alpha k+1 }= (\rho^{\alpha})^{-1} g^{\alpha k+1} \rho^{k+1} \text{ on } U_{\alpha, k+1} \cap U_{k+1,k+1} \text{ for } \alpha \leq k+1 .$$
	Then there exists $C>0$ depending only on $G$ such that
	$$\Vert \rho_{k+1}-\mathrm{id}\Vert_\infty \leq C\Vert \exp^{-1}(\rho^{k+1})\Vert_\infty\leq C  \Vert u_{k+1} \Vert_{\mathrm{L}^\infty} \leq 2C(C_k +1)m,$$
	hence we can set $C_{k+1}=2C(C_k +1)$. Moreover
	\begin{align*}
		\Vert \rho_{k+1}-\mathrm{id}\Vert_{\mathrm{W}^{2,2}} &\leq C\Vert \exp^{-1}(\rho^{k+1})\Vert_{\mathrm{W}^{2,2}} \leq C \sum_{\alpha\leq k} \Vert \phi_{k+1}\exp^{-1}(g^{k+1,\alpha}\rho^\alpha h^{\alpha,k+1})\Vert_{\mathrm{W}^{2,2}(U_{k+1}\cap U_{\alpha,k} )} \\
		&\leq \Vert u_{k+1}\Vert_\infty P( (\Vert g^{k+1\alpha} \Vert_{\mathrm{W}^{2,2} (U_\alpha\cap U_{k+1})})_{\alpha\leq k},(\Vert h^{\alpha k+1} \Vert_{\mathrm{W}^{2,2} (U_\alpha\cap U_{k+1})})_{\alpha\leq k}, (\Vert \rho^\alpha-\mathrm{id} \Vert_{\mathrm{W}^{2,2}})_{\alpha\leq k}). 
	\end{align*}
	where $P$ is polynomial expression depending only on the $\Vert \phi_{k+1}\Vert_{\mathrm{W}^{2,2}} $ which depends only on the covering $\{U_\alpha\}_{\alpha\in I}$.Hence thanks to \eqref{Pk}, there exists a polynomial $P_{k+1}$ depending only on $G$ and $\{U_\alpha\}_{\alpha\in I}$ such that 
	$$	\Vert \rho_{k+1}-\mathrm{id}\Vert_{\mathrm{W}^{2,2}}  \leq m P_{k+1}( (\Vert g^{ab} \Vert_{\mathrm{W}^{2,2} (U_a\cap U_b)})_{a,b\leq k+1},(\Vert h^{ab} \Vert_{\mathrm{W}^{2,2} (U_a\cap U_b)})_{a,b\leq k+1}),
	$$
	which achieves the proof.
\end{proof}

The key ingredient of the theory of weak-connections is the following result.
\begin{theorem}[Theorem A \cite{PR14}] \label{controlledglobalgauge}
    Let $M$ be a Riemannian $4$-manifold. There exists $f: \R^+ \rightarrow \R^+$ with the following property: Let $\nabla$ be a $\mathrm{W}^{1,2}$-connection over a principal $\mathcal P^{2,2}_G$-bundle over $M$. Then there exists a global $\mathrm{W}^{1,(4,\infty)}$-section of the bundle (possibly allowing singularities) over the whole $M$ such that in the corresponding trivialization $\nabla$ is given by $d+A$ with the bound 
    $$\Vert A\Vert_{\mathrm{L}^{4,\infty}} \leq f(\Vert F\Vert_{\mathrm{L}^2}),$$
    where $F$ is the curvature form of $\nabla$.
\end{theorem}

For the sake of completeness, we remind the definition of the weak-$\mathrm{L}^p$ space, see section section 1.4 of \cite{Grafakos1} for details. Let $(X,\mu)$ be a measured space, the space $\mathrm{L}^{p,\infty}(X,\mu)$ is the space of all measurable functions $f$ such that $$\Vert f\Vert^p_{\mathrm{L}^{p,\infty}} := \sup_{\lambda >0} \lambda^p \mu (\{x \, \vert \, \vert f(x)\vert >\lambda \})$$
is finite. In particular $\frac{1}{\vert x\vert} \in \mathrm{L}^{4,\infty}(\mathrm{B}^4)\setminus \mathrm{L}^4(\mathrm{B}^4)$.\\ 

Hence to each $\mathrm{W}^{1,2}$-connection we can associate a \textbf{global} one form into $\mathrm{L}^{4,\infty}$. Which leads us to define the space of weak-connections
$$ \mathfrak U_G(M):=\left\{ \begin{array}{l}
A\in \mathrm{L}^{4,\infty}(M,\Lambda^1 T^* M\otimes \mathfrak g) \text{ s.t. } dA+A\wedge A \in \mathrm{L}^2(M,\Lambda^2T^*M\otimes \mathfrak g) \\
    \text{and locally } \exists g \in \mathrm{W}^{1,(4,\infty)} \text{ s.t. } A^g\in \mathrm{W}^{1,2}.
\end{array}\right\}$$
At first glance, it seems weird to get a global object for a connection on a principal bundle which \textit{ a priori } gets some topology. But as we will see in the following example all the topology is encoded in the singularities of the connection.\\

\textbf{ The example of instantons:} If we consider the quaternionic Hopf's bundle over $\mathbb HP^1\cong S^4$ with fiber $Sp(1)\cong \mathrm{SU}(2)$. Then, over $\mathbb H \cong \mathbb HP^1 \setminus \{[0,1]\}$, the connection is given by 
$$ \tilde A=\Im \left( \frac{\bar{q}dq}{1+\vert q\vert^2}\right).$$
Now if we set $A=\pi^*(\tilde A): S^4 \rightarrow \Lambda^1 T^* S^4 \otimes \mathfrak{su} (2)$, where $\pi$ is the stereographic projection with respect to the north pole. It defines a singular connection form with a singularity at north pole like the inverse of the distance to the north pole. Hence $\tilde A \in \mathrm{L}^{4,\infty} (S^4, \Lambda^1 T^* S^4 \otimes \mathfrak{su} (2))$ but not in $\mathrm{L}^4$. Moreover, since,
$$A(q) \underset{+\infty}{\sim}\Im \left( \left( \frac{q}{\vert q\vert }\right)^{-1} \left( d\frac{q}{\vert q\vert }\right)\right).$$
In the chart given by stereographic projection with respect to south pole, $A^g$ is a smooth connection where $g=\frac{q}{\vert q\vert} \in \mathrm{W}^{1,(4,\infty)}(S^4, \mathrm{SU}(2))$.\\

More generally, we can associate to each weak connection a $\mathrm{W}^{1,2}$-connection and vice-versa.
\begin{theorem}[Proposition 1.4 \cite{PETRACHE2017469}] Let $M$ be a compact Riemannian 4-manifold. There exists a surjective geometric realization map
$$\mathfrak R : \mathfrak U_G(M) \rightarrow   \hat{\mathcal P}^{2,2}(M),$$
%such that if $\mathcal R (A)=[(P,\tilde A)]_{2,2}$ then for representatives $P=(\{ U_\alpha\}_{\alpha\in I}, \{g^{\alpha\beta}\}_{\alpha,\beta \in I} ) $ and $\tilde A=  \{ A_\alpha\}_{\alpha\in I}$, we have that on each $U_\alpha$ there exists a change of gauge map $g_\alpha \in \mathrm{W}^{1,(4,\infty)}(U_\alpha, G)$ such that $A_\alpha=A^{g_\alpha}$. 
\end{theorem}

In fact, we can even associated to a weak-connection a $\mathrm{W}^{2,2}\cap C^0$-bundle, see proposition VIII.1 of \cite{rivière2015variations}, choosing the gauge $g_\alpha$ being given by the following version of Uhlenbeck's Coulomb gauge extraction. Indeed, as shown in proof of lemma \ref{globalg}, the $g^{\alpha\beta}(g^\alpha)^{-1}g^\beta$ given by the second part of the following theorem are control in $C^{0}\cap \mathrm{W}^{2,2}.$

\begin{theorem}
\label{extract} Let $A\in \mathfrak U_G(M)$ there exists $\epsilon_G>0$ and $C>0$ such that if $U$ is an open set diffeomorphic to a ball on which there exists $g \in \mathrm{W}^{1,(4,\infty)}(U,G)$ such that  $A^g\in \mathrm{W}^{1,2}(U,\mathfrak g)$ and 
$$\Vert (dA+A\wedge A) \Vert_{\mathrm{L}^2(U)} \leq \epsilon_G $$
then there exists $g' \in \mathrm{W}^{1,(4,\infty)}(U,G)$ such that 
$$\diff^* A^{g'}=0$$
and
$$\Vert A^{g'}\Vert_{\mathrm{W}^{1,2}(U)} \leq C \Vert (dA+A\wedge A)\Vert_{\mathrm{L}^2(U)}.$$
In particular, there exists a good cover $\{ U_\alpha\}_{\alpha\in I}$ of $M$ and $g_\alpha \in \mathrm{W}^{1,(4,\infty)}(U_\alpha,G)$ such that 
$$\diff^* A^{g_\alpha}=0$$
and
$$\Vert A^{g_\alpha}\Vert_{\mathrm{W}^{1,2}(U_\alpha)} \leq C \Vert (dA+A\wedge A) \Vert_{\mathrm{L}^2(U_\alpha)}.$$
\end{theorem}
\begin{remark}
As shown below, the proof consists into applying the classical result on $\mathrm{W}^{1,2}$-connection locally to $A^g\in \mathrm{W}^{1,2}$. More generally any result with is true for $\mathrm{W}^{1,2}$-connection up to $\mathrm{W}^{2,2}$-gauge is true for weak-connections up to $\mathrm{W}^{1,(4,\infty)}$-gauge.
\end{remark}
\begin{proof}
For the first part of the theorem it suffices to apply the Coulomb gauge extraction on $\mathrm{W}^{1,2}$-connection to $A^g$, see theorem IV.1 of \cite{rivière2015variations}. For the second,  since by definition of the weak connection we can find some good cover $\{ U_\alpha\}_{\alpha\in I}$ of $M$ and $g_\alpha \in \mathrm{W}^{1,(4,\infty)}(U_\alpha,G)$ such that $ A^{g_\alpha}\in \mathrm{W}^{1,2}$. Then up to take a refinement we can also assume that $\Vert F \Vert_{\mathrm{L}^2(U_\alpha)}\leq \epsilon_G$, where $\epsilon_G$ is the one of the Coulomb gauge extraction for $\mathrm{W}^{1,2}$-connection. Then applying the first part of the theorem we obtain the desired gauges.
\end{proof}

Thanks to lemma \ref{Ulem}, we are in fact able to build a global gauge for any weak-connection on a ball without the assumption that the curvature is small. However we can't get a uniform control but we prove that if the curvature is small on a dyadic annulus then we recover uniform estimate on this annulus.

\begin{lemma}	 Let $A\in \mathfrak U_G(B)$ then there exists $g\in \mathrm{W}^{1,(4,\infty)}(\mathrm{B}_\frac{3}{4},G)$ such that 	$$A^g \in \mathcal A^{1,2}_G(\mathrm{B}_\frac{1}{2}).$$
	Moreover there exists $\epsilon_G'>0$ depending only on $G$, such that  if there exists $0<r<1/2$ such that
	$$ \Vert F_A\Vert_{\mathrm{L}^2(\mathrm{B}_{2r}\setminus \mathrm{B}_{r/4})} \leq \epsilon_G',$$
	then the gauge can be chosen such that
	$$\Vert A^g \Vert_{\mathrm{W}^{1,2}(\mathrm{B}_{r}\setminus \mathrm{B}_{r/2}))} \leq C  \Vert F_A\Vert_{\mathrm{L}^2(\mathrm{B}_{2r}\setminus \mathrm{B}_{r/4})}$$
	with $C>0$ depending only on $G$. \label{globalg}
\end{lemma} 
The begining of the proof is inspired by the proof of theorem V.5 of \cite{rivière2015variations}.

\begin{proof}
	We start by covering $\mathrm{B}_\frac{3}{4}$ by balls $(B(x_\alpha,r_\alpha))_{\alpha\in\{1,\dots,N\}}$ such that we have 
	$$\int_{B(x_\alpha,r_\alpha)} \vert F_A\vert^2_h \, \mathrm{vol}_h <\epsilon<\epsilon_G,$$
	where $\epsilon_G$ is the one given by theorem \ref{extract} and $\epsilon$ will be fixed later but will only depend on $G$. Hence for each $\alpha$ there exists $g^\alpha\in \mathrm{W}^{1,(4,\infty)}(B(x_\alpha,r_\alpha),G)$ such that $A^{g^\alpha}\in \mathrm{W}^{1,2}(B(x_\alpha,r_\alpha),\mathfrak g)$ and
	\begin{equation}
		\label{CoulombW12}
		\Vert A^{g^\alpha}\Vert_{\mathrm{W}^{1,2}(B(x_\alpha,r_\alpha))} \leq C_G \Vert F_A\Vert_{\mathrm{L}^2(B(x_\alpha,r_\alpha))}.
	\end{equation}
	
	Let $g^{\alpha\beta}=(g^\alpha)^{-1}g_\beta $ on $B(x_\alpha,r_\alpha) \cap B(x_\beta,r_\beta)$, then it satisfies
	$$d g^{\alpha\beta}=A^{g^\alpha}g^{\alpha\beta}-g^{\alpha\beta}A^{g^\beta},$$
	hence $g^{\alpha\beta}$ is controlled in $\mathrm{W}^{2,2}$, moreover, since $d* A^{g^{\alpha}}=0$ for all $\alpha$, we have
	$$-\Delta_h g^{\alpha\beta}=A^{g^\alpha}. dg^{\alpha\beta}-dg^{\alpha\beta}.A^{g^\beta}.$$
	Using the injection $\mathrm{W}^{1,2}.\mathrm{W}^{1,2}\hookrightarrow \mathrm{L}^{4,2}.\mathrm{L}^{4,2}\hookrightarrow \mathrm{L}^{2,1}$ and Calderon-Zygmund theory, we obtain 
	\begin{equation}
		\label{W22g}
		\Vert\nabla^2 g^{\alpha\beta}\Vert_{\mathrm{L}^{2,1}(B(x_\alpha,\frac{3}{4}r_\alpha)\cap B(x_\beta,\frac{3}{4}r_\beta))} \leq C \Vert F_A\Vert_{\mathrm{L}^2(B(x_\alpha,r_\alpha)\cup B(x_\beta,r_\beta))}.
	\end{equation}
	Hence, using the injection\footnote{It is a consequence, for $n=4$, of the Peetre-Sobolev embedding $\mathrm{W}^{2,(n/2,1)} \hookrightarrow \mathrm{W}^{1,(n,1)}$, see XX, combined with the Stein embedding  $\mathrm{W}^{1,(n,1)}\hookrightarrow C^0$, see \cite{Stein}.}\  $\mathrm{W}^{2,(2,1)} \hookrightarrow C^0$, we get that $g^{\alpha\beta}\in C^0(B(x_\alpha,\frac{3}{4}r_\alpha)\cap B(x_\beta,\frac{3}{4}r_\beta))$ and there exists $\overline{g}^{\alpha\beta}\in G$ such that 
	\begin{equation}
		\label{estimg}
		\Vert g^{\alpha\beta}-\overline{g}^{\alpha\beta}\Vert_{\mathrm{L}^\infty (B(x_\alpha,\frac{3}{4}r_\alpha)\cap B(x_\beta,\frac{3}{4}r_\beta))} \leq C \Vert F_A\Vert_{\mathrm{L}^2(B(x_\alpha,r_\alpha)\cup B(x_\beta,r_\beta))}.
	\end{equation}
	Moreover, we can choose the $\overline{g}^{\alpha\beta}$ such as they satisfy the cocyle condition. Hence since the ball is topologically trivial, the $\overline{g}^{\alpha\beta}$  defines a trivial \v{C}ech constant co-chain for the presheaf of $G$-valued constant function. Then, see section 10 of Chapter II of \cite{BottTu}, there exists $\overline{\rho}^\alpha \in G$ such that
	$$\overline{g}^{\alpha\beta}=(\overline{\rho}^\alpha)^{-1}\overline{\rho}^\beta \text{ for all } \alpha,\beta \in I.$$ 
	There exists $\epsilon=\epsilon_G'$ such that, thanks to \eqref{estimg}, we can apply lemma \ref{Ulem} to $g^{\alpha\beta}$ and $\overline{g}^{\alpha\beta}$ , hence there exists a refinement $\{V_\alpha\}$ of $\{U_\alpha\}$ and $\rho^\alpha \in \mathrm{W}^{2,2}\cap C^0(V_\alpha, G)$ such that 
	$$g^{\alpha\beta}=(\rho^\alpha)^{-1}\overline{g}^{\alpha\beta}\rho^\beta \text{ on } V_\alpha \cap V_\beta$$
	and
	$$ \Vert \rho^{\alpha} -\mathrm{id} \Vert_{\mathrm{L}^\infty (V_\alpha)} \leq C \sup_{\alpha,\beta}\Vert F_A\Vert_{\mathrm{L}^2(B(x_\alpha,r_\alpha)\cup B(x_\beta,r_\beta))}$$
	and 
	\begin{equation}
		\label{rhoW22}
		\Vert \rho^{\alpha}-\mathrm{id} \Vert_{\mathrm{W}^{2,2} (V_\alpha)} \leq  P_\alpha \sup_{\alpha,\beta} \Vert F_A\Vert_{\mathrm{L}^2(B(x_\alpha,r_\alpha)\cup B(x_\beta,r_\beta))}. 
	\end{equation}
	where $P_\alpha$ is a polynomial expression in $(\Vert g^{ab} \Vert_{\mathrm{W}^{2,2} (U_a\cap U_b)})_{a,b\leq \alpha})$ and $(\Vert h^{ab} \Vert_{\mathrm{W}^{2,2} (U_a\cap U_b)})_{a,b\leq \alpha})$ whose coefficient depends only on $\{U_a\}_{a\leq \alpha}$.
	Hence  we set $\tilde{g}^\alpha=g^\alpha(\tilde{\rho}^\alpha)^{-1}$ with $\tilde{\rho}^\alpha=\overline{\rho}^\alpha\rho^\alpha$ and we easily check that it defines an element of $\tilde g\in \mathrm{W}^{1,(4,\infty)}(\mathrm{B}_\frac{3}{4})$ such that 
	$$A^{\tilde g}=(A^{g^\alpha})^{(\tilde \rho^\alpha)^{-1}}=\tilde \rho^\alpha d(\tilde \rho^\alpha)^{-1}+\tilde\rho^\alpha A^{g^\alpha}(\tilde \rho^\alpha)^{-1} \text{ on }V_\alpha.$$
	Then thanks to the estimates on $\rho^\alpha$ we have the first conclusion.\\ 
	To get the second assertion, we choose $\epsilon_G' >0$ as above and we start  by covering $(\mathrm{B}_{3r/2}\setminus \mathrm{B}_{r/3})$ by  $(B(x_\alpha,r_\alpha))_{\alpha\in\{1,\dots,N_0\}}$ such that we have 
	\begin{equation}
		\label{Fcontrol}
		\int_{B(x_\alpha,r_\alpha)} \vert F_A\vert^2_h \, \mathrm{vol}_h <\epsilon_G'.
	\end{equation}
	Since the region  $(\mathrm{B}_{3r/2}\setminus \mathrm{B}_{r/3})$ is dyadic we have a universal control on $N_0$ the number of balls covering  $(\mathrm{B}_{3r/2}\setminus \mathrm{B}_{r/3})$. Then we complete the cover by balls disjoint to  $(\mathrm{B}_{3r/2}\setminus \mathrm{B}_{r/3})$ and we proceed exactly as in the proof of the first assertion. It is very important to note that the covering $(B(x_\alpha,r_\alpha))_{\alpha\in\{1,\dots,N_0\}}$ is universal so are the coefficient of $P_\alpha$ for $\alpha\leq N_0$. Indeed, if we look to the proof of lemma \ref{Ulem} the coefficient will only depend on the $\mathrm{W}^{2,2}$ norm of cut-off functions on ball which can be uniformly controlled since the $\mathrm{W}^{2,2}$ norm is invariant  by scaling in dimension $4$. Moreover, thanks to \eqref{W22g} and \eqref{Fcontrol}, each variable of $P_\alpha$ is assume to be control by 
	a constant depending only on $G$, then $P_\alpha$ can be replace by a constant depending only on $G$ for $\alpha \leq N_0$. Thanks to the control of the number of balls, \eqref{CoulombW12}, \eqref{rhoW22} and the previous remark we obtain the desired conclusion.
\end{proof}

We can now  define the following notion of convergence which is compatible with the notion of weak-connection. 
\begin{propdef}
    A sequence $A_k \in  \mathfrak U_G(M)$ converges into $\mathfrak U_G(M)$  to $A_\infty \in \mathfrak U_G(M)$ if there exists a good cover $\{ U_\alpha\}_{\alpha\in I}$ of $M$ and $g^{\alpha}_k \in \mathrm{W}^{1,(4,\infty)}(U_\alpha,G)$ such that 
$$A_k^{ g^{\alpha}_k} \rightarrow A_\alpha^\infty \text{ in } \mathrm{W}^{1,2}(U_\alpha),$$
and $\mathfrak R(A_\infty) =[(P, \{A_\alpha^\infty\})]_{2,2}$.\\
We say that the convergence is weak into $\mathfrak U_G$, if we have only $A_k^{ g^{\alpha}_k} \rightharpoonup A_\alpha^\infty \text{ in } \mathrm{W}^{1,2}(U_\alpha)$.
Finally, we say that  $A_k$  converge to $A$ into $\mathfrak C_G^\infty (M)$, if we have only
$$A_k^{ g^{\alpha}_k} \rightharpoonup A_\alpha^\infty \text{ in } \mathcal{C}^\infty(U_\alpha).$$
Finallly, for $U\subset M$ an open set, we note $\mathfrak C^\infty_{G,\mathrm{loc}} (U)$ if the convergence holds in $\mathfrak C^\infty_{G} (K)$ for any compact set $K\subset U$.  
\end{propdef}
\begin{proof}
First of all the $A_\alpha^\infty$ defined a connection. Indeed, on $U_\alpha \cap U_\beta$ we have
	$$dg^{\alpha\beta}_k=g^{\alpha\beta}_k A_k^{g^{\beta}_k}-A_k^{g^{\alpha}_k} g^{\alpha\beta}_k,$$
	where $g^{\alpha\beta}_k=(g^{\alpha}_k)^{-1} g^{\beta}_k$. Hence $g^{\alpha\beta}_k$, up to a subsequence, weakly converge in $\mathrm{W}^{2,2}$ to $g^{\alpha\beta}_{\infty}$ which define a  $\mathcal P_G^{2,2}$-bundle on which $A_\alpha^\infty$ defines a $\mathrm{W}^{1,2}$-connection.\\
	Then only thing we have to check is that the limit does not depends on the cover and gauges, i.e. the limit weak-connection define a unique $\mathcal P^{2,2}$-bundle with a given $\mathrm{W}^{1,2}$-connection. Indeed, if we another couple  good cover $\{V_\alpha\}$ with gauge $\{h^{\alpha}_k\}$ such that $A_k^{h_k^{\alpha}}$ converges in $\mathrm{W}^{1,2}$ to $B_\alpha^\infty$. Then $h^{\alpha\beta}_k=(h^{\alpha}_k)^{-1} h^{\beta}_k$, up to a subsequence, converge weakly in $\mathrm{W}^{2,2}$  to $h^{\alpha\beta}_{\infty}$ which defines a $\mathcal P^{2,2}$-bundle for which $B_\alpha^\infty$ define a $\mathrm{W}^{1,2}$-connection. Then we have $\{W_\alpha\}$ a good cover which is a refinement of  $\{U_\alpha\}$  and $\{V_\alpha\}$, and gauges still denoted $\{g^{\alpha}_k\} $ and $\{h^{\alpha}_k\} $ which are restrictions of the initial $\{g^{\alpha}_k\} $ and $\{h^{\alpha}_k\} $, such that 
\begin{equation} 
	\label{ghk}
	A_k^{g^{\alpha}_k}=(k^{\alpha}_k)^{-1} dk^{\alpha}_k+(k^{\alpha}_k)^{-1}A_k^{h^{\alpha}_k} k^{\alpha}_k \text{ on } W_\alpha,
\end{equation}	
where $k^{\alpha}_k={(g^{\alpha}_k)}^{-1} h^{\alpha}_k.$ Hence, $k^{\alpha}_k$, up to a subsequence, weakly convergences in $\mathrm{W}^{2,2}$ to $k^{\alpha}_{\infty}$ and 
$$g^{\alpha\beta}_\infty= k^{\alpha}_{\infty} h^{\alpha\beta}_\infty( k^{\beta}_\infty)^{-1} \text{ a.e. on } W_\alpha\cap W_\beta.$$
So the two bundles are $\mathrm{W}^{2,2}$-equivalent, and more over passing to the limit in \eqref{ghk} we see that $A_\alpha^\infty$ and $B_\alpha^\infty$ define the same connection.
\end{proof}
	A natural question is: Is this notion of convergence compatible with the classical one, i.e. $A_k \rightarrow A_\infty$ in $\mathrm{L}^{4,\infty}(M)$? Yes! As seen in the previous proof, $g^{\alpha\beta}_k=(g^{\alpha}_k)^{-1} g^{\beta}_k$ converges in $\mathrm{W}^{2,2}$ to $g^{\alpha\beta}_\infty$ which define a  $\mathcal P_G^{2,2}$-bundle on which $A_\alpha^\infty$ defines a connection.	Moreover, this connection is equal to $\mathfrak{R}(A_\infty)$ since on each $U_\alpha$ we have
	$$dg^{\alpha}_k=g^{\alpha}_k A_k^{g^{\alpha}_k}-A_k g^{\alpha}_k,$$
	then we easily see that, up to a subsequence, $g^{\alpha}_k$  converge weakly to $g_\alpha^\infty$ in $\mathrm{W}^{1,(4,\infty)}$ which gives that $A_\infty^{g^{\alpha}_\infty}=A_\alpha^\infty$ a.e..\\

However, as we will see, due to concentration phenomena, the convergence is sometimes only strong outside a finite set of points. We need to extend the connection through this points, this is the goal of the following point removability result.\\

\begin{theorem}Let $A\in \mathfrak U_G(U\setminus\{p\})$ where $U\subset M$ is an open set diffeomorphic to a ball and $p\in U$, then there exists $g \in \mathrm{W}^{1,(4,\infty)}(U,G)$ such that  $A^g\in \mathrm{W}^{1,2}(U,\mathfrak g)$.
\end{theorem}

\begin{proof} By hypothesis, we have $dA+A\wedge A \in \mathrm{L}^2(U, \Lambda^2 T^*U\otimes \mathfrak g )$, hence we can take a small ball around $p$ on which the total curvature does not exceed $\epsilon_G$ and then proceed as in theorem V.6 of \cite{rivière2015variations} to find a gauge $g \in \mathrm{W}^{1,(4,\infty)}(U,G)$ such that  $A^g\in \mathrm{W}^{1,2}(U,\mathfrak g)$.
\end{proof}

For this notion of convergence, we have the following weak compactness result. It is a translation in our setting  of theorem VII.2 and VIII.1 of \cite{rivière2015variations}. 

\begin{theorem} Let $A_k\in \mathfrak U_G(M)$ such that
$$ \Vert dA_k + A_k \wedge A_k\Vert_2 \leq C,$$
then there exists $A_\infty \in \mathfrak U_G(M)$ such that, up to a subsequence, $A_k$ weakly converges to $A_\infty$ into   $\mathfrak U_G(M)$. If moreover the connection is Yang-Mills then the convergence holds in $\mathfrak C^\infty_{G,\mathrm{loc}}(M\setminus\{ \text{finitely many points} \}$.  
\end{theorem}

\begin{proof}
    Let $\epsilon_G >0$ given by Coulomb gauge extraction, then there exists at most finitely many points $\{p_1,\dots, p_N\}$ where the $ \Vert dA_k + A_k \wedge A_k\Vert_2$ concentrates above $\epsilon_G$, then thanks to theorem \ref{extract}, we can cover any compact subset of $M\setminus \{p_1,\dots, p_N\}$ by a good cover where we have on each open set, up to gauge, a $\mathrm{W}^{1,2}$-control on the connection. Then we easily conclude by diagonal argument. In the case of Yang-Mills connection the equation in the Coulomb gauge becomes elliptic and then we have $\mathcal{C}^\infty_{\mathrm{loc}}$ convergence, up to gauge. 
\end{proof}

Finally we reformulate the quantization phenomena in our framework, see theorem I.2 \cite{riviere2002interpolation} and theorem VII.3 \cite{rivière2015variations} or theorem II.7 \cite{memoireGauvrit}.

\begin{theorem}
%\label{bubbletreebis}
    Let $ (M,h)$ be a closed $4$-dimensional Riemannian manifold and  $A_k\in\mathfrak U_G(M^4)$ a sequence of Yang-Mills connection with uniformly bounded energy. Then we have, up to a sub-sequence,
    \begin{enumerate}
        \item There exists finitely many points $\{p_1,\dots, p_N\}$ and a Yang-Mills connection $A_\infty\in \mathfrak U_G(M^4)$ such that $A_k$ converges to $A_\infty$ in $\mathfrak C^\infty_{G,\mathrm{loc}}(M\setminus\{p_1,\dots, p_N\}$.
        \item For each $i\in {1, \dots, N}$, there exists $N_i\in \N$ sequences of points $(p_k^{i,j})_k$ converging to $p_i$,  $N_i$ sequences of scalars $(\lambda_k^{i,j})_k$ converging to zero, $N_i$ non-trivial Yang-Mills connections $A_\infty^{i,j} \in\mathfrak U_G(S^4)$, called bubbles, such that,  
        $$(\phi_{k}^{i,j})^*(A_k) \rightarrow \widehat{A}_\infty^{i,j}=\pi_*A_\infty^{i,j}  \text{ into } \mathfrak C^\infty_{G,\mathrm{loc}}(\R^4\setminus S_{i,j}),$$
        where $\phi_{k}^{i,j} (x)= p_k^{i,j} +\lambda _k^{i,j} x$ in local coordinates, $\pi$ is the stereographic projection, and $\displaystyle S_{i,j}=\lim_{k\rightarrow \infty }\left\{ \frac{p_k^{i,l}-p_k^{i,j}}{\lambda_k^{i,j}} \text{ for } l\not=j\right\}$.
        \item Moreover there is no loose of energy, i.e
        $$ \lim_{k\rightarrow +\infty}\int_{M^4} \vert F_{A_k}\vert^2_h\, \mathrm{vol}_h  = \int_{M^4} \vert F_{A_\infty}\vert^2_h\, \mathrm{vol}_h +\sum_{i=1}^N \sum_{j=1}^{N_i} \int_{S^4} \vert F_{A_\infty^{i,j}}\vert^2\, \mathrm{vol}. $$
    \end{enumerate}
\end{theorem}
Let us clarify the last point of the theorem in the generic situation where there is only one bubble concentrating at the center of the ball $\mathrm{B}^4$. So we have $A_k\in \mathfrak U_G(\mathrm{B}^4)$ such that $A_k$ converge to $A_\infty \in \mathfrak U_G(\mathrm{B}^4)$ in $\mathfrak C_{G,\mathrm{loc}}(\mathrm{B}^4\setminus \{0\})$ and there exists $\lambda_k\rightarrow 0$ and $A_\infty^1\in \mathfrak U_G(S^4)$ such that
$$\phi_k^*(A_k) \rightarrow \widehat{A}_\infty^1=\pi_*A_\infty^1 \text{ into } \mathfrak C_{G,\mathrm{loc}}^\infty(\R^4)$$
where $\phi_k(x)=\lambda_k x$ and $\pi$ is the stereographic projection. Then the last point of the theorem means that there is no loose of energy in the neck region, i.e. the annular whose inner radius is big with respect to $\lambda_k$ and outer radius small with respect to $1$. That is to say,
$$\lim_{\eta \rightarrow 0} \lim_{k\rightarrow +\infty} \Vert (dA_k+A_k\wedge A_k)\Vert_{\mathrm{L}^2 (B(0,\eta)\setminus B(0,\lambda_k/\eta)}=0.$$
\subsection{Functional analysis results}

\begin{lemma} \label{projectionstereosobo}
 The stereographic projection $\pi : \mathbf{S}^4\setminus\{\text{north pole}\} \rightarrow
  \R^4$ induces an isometry $\pi^{\ast} : \dot{\mathrm{W}}^{1, 2}
  (\mathbf{S}^4,T^*\mathbf{S}^4\otimes\mathfrak{g}) \rightarrow \dot{\mathrm{W}}^{1, 2}
    (\R^4,T^*\R^4\otimes \mathfrak{g})$.
\end{lemma}
\begin{proof}
  Similar to \cite[proposition IV.1]{daliorivieregianocca2022morse}.
\end{proof}

\subsubsection{Hardy-Poincaré type inequalities in neck regions}

\begin{proposition} There exists $C>0$ such that for all $R>r>0$, for all $\mathrm{W}^{1, 2}_0 \left( \mathrm{B}_R \backslash \mathrm{B}_r\right)$-valued $1$-form $a$,
\[ \int_{\mathrm{B}_R \backslash \mathrm{B}_r} \frac{| a |^2}{| x |^2} \diff x \leqslant C \| \nabla
   a \|_{\mathrm{L}^2 \left( \mathrm{B}_R \backslash \mathrm{B}_r \right)}^2 \] \label{hardyneck}
\end{proposition}

\begin{proof}
Using $\mathrm{L}^{4, 2} \cdot \mathrm{L}^{4, \infty} \hookrightarrow
\mathrm{L}^2$ and $\dot{\mathrm{W}}_0^{1, 2} \hookrightarrow \mathrm{L}^{4, 2}$, we get \begin{align*}
    \int_{\mathrm{B}_R \backslash \mathrm{B}_r} \frac{| a |^2}{| x |^2} \diff x =
   \left\| \frac{1}{| x |} a \right\|_{\mathrm{L}^2 \left( \mathrm{B}_R \backslash
   \mathrm{B}_r \right)}^2 &\leqslant \left\| \frac{1}{| x |}
   \right\|^2_{\mathrm{L}^{4, \infty} (\R^4)} \| a \|_{\mathrm{L}^{4, 2}
   \left( \mathrm{B}_R \backslash \mathrm{B}_r \right)}^2 \\
   & \leqslant C \left\|
   \frac{1}{| x |} \right\|^2_{\mathrm{L}^{4, \infty} (\R^4)} \| \nabla
   a \|_{\mathrm{L}^2 \left( \mathrm{B}_R \backslash \mathrm{B}_r \right)}^2 .
\end{align*} We conclude thanks to
\[ \left\| \frac{1}{| x |} \right\|_{\mathrm{L}^{4, \infty} (\R^4)} =
   \sup_{t > 0} t | \{ x \in \R^4, | x |^{- 1} > t \} |^{1 / 4} =
   \sup_{t > 0} t \left| \mathrm{B}_{1 / t} \right|^{1 / 4} = \left| \mathrm{B}_1
   \right|^{1 / 4} < + \infty . \]
\end{proof}

\begin{proposition} \label{poincareneck}
   There exists $C>0$ such that for all $R>r>0$, for all $ \mathrm{W}^{1, 2}_0 \left( \mathrm{B}_R \backslash \mathrm{B}_r\right)$-valued $1$-form $a$  \begin{align*}
        \int_{\mathrm{B}_R \backslash \mathrm{B}_r} | a |^2 \diff x &\leqslant C R^2 
   \| \nabla a \|_{\mathrm{L}^2 \left( \mathrm{B}_R \backslash \mathrm{B}_r
   \right)}^2, \\
  \int_{\mathrm{B}_R \backslash \mathrm{B}_r} \frac{| a |^2}{|x|^4}\diff x  &\leqslant \frac{C}{r^2} \| \nabla a \|_{\mathrm{L}^2 \left( \mathrm{B}_R
   \backslash \mathrm{B}_r \right)}^2.
    \end{align*}
\end{proposition}

\begin{proof} The first inequality is true for $R=1$ and $r\in]0,1[$ (using Poincaré inequality), then we apply it to
$m_R^{\ast} a \in \mathrm{W}^{1, 2}_0 \left( \mathrm{B}_1 \backslash \mathrm{B}_{r /
R} \right)$ where $m_R : x \mapsto R x$).

\noindent We prove the second inequality using the conformal invariance of the right-hand side and the first one : introduce $b =
\varphi^{\ast} a$ (where $\varphi : x \mapsto x / | x |^2$ is the inversion), we have
$b \in \mathrm{W}^{1, 2}_0 \left( \mathrm{B}_{1 / r} \backslash \mathrm{B}_{1/R}
\right)$ so
\[ \int_{\mathrm{B}_{1 / r} \backslash \mathrm{B}_{1 / R}} | b |^2 \diff x
   \leqslant \frac{C}{r^2}  \| \nabla a \|_{\mathrm{L}^2 \left( \mathrm{B}_R
   \backslash \mathrm{B}_r \right)}^2. \]
Then, since $\varphi^{\ast} \diff x = \frac{1}{| x |^8} \diff x$, we get \begin{align*}
\int_{\mathrm{B}_{1 / r} \backslash \mathrm{B}_{1 / R}} | b |^2 \diff x =
   \int_{\mathrm{B}_{1 / r} \backslash \mathrm{B}_{1 / R}} \left( \frac{1}{| x
   |^2} \right)^2 | a |^2 \circ \varphi \,\diff x &= \int_{\mathrm{B}_{1 / r}
   \backslash \mathrm{B}_{1 / R}} \varphi^{\ast} \left( \frac{1}{| x |^4} | a
   |^2 \diff x \right)\\ &= \int_{\mathrm{B}_R \backslash \mathrm{B}_r} \frac{| a
   |^2}{| x |^4} \diff x.    
\end{align*}
\end{proof}

\subsubsection{Gaffney inequality}

\begin{lemma}\label{gaffney} Let $n\in \N^*$.
   For all $p\in ]1;+\infty[$, $q\in ]0;+\infty] $, there exists a constant $C_{p,q}>0$ such that for all differential forms $\omega$ on $\R^n$, \[\|\nabla\omega\|_{\mathrm{L}^{p,q}(\R^n)}\leqslant C_{p,q} \left(\|\diff \omega\|_{\mathrm{L}^{p,q}(\R^n)}+ \|\diff^{\ast} \omega\|_{\mathrm{L}^{p,q}(\R^n)} \right).\]
\end{lemma}

\begin{proof}
    See for instance \cite{troyanov2010hodge}.
\end{proof}

\subsection{Diagonalization of quadratic forms} \label{Diagformquad}

\subsubsection{Spectral theory results}

\begin{lemma} \label{formequaddiagonalisée}
  Let $(\mathcal{L}, D (\mathcal{L}))$ a self-adjoint operator, bounded below, on a Hilbert space $\mathfrak{H}$. Denote by $q$ the corresponding closed quadratic form with domain $\mathcal{Q}$. Assume that $\mathcal{L}$ has compact resolvent and denote by $(e_n)_{n \in \N}$ an Hilbert basis of eigenvectors, with corresponding eigenvalues $(\lambda_n)_{n \in \N}$ (in ascending order). Then
  \[ \mathcal{Q}= \left\{ x \in \mathfrak{H}; \sum_{n \in \N} |
     \lambda_n | | \langle e_n, x \rangle |^2 < + \infty \right\} \]
  and if $x \in \mathcal{Q}$, $q (x) = \sum_{n \in \N} \lambda_n |
  \langle x, e_n \rangle |^2$.
\noindent Furthermore,
  \[ \mathcal{Q}= \bigoplus_{\lambda \leqslant 0} \ker (\mathcal{L}- \lambda)
     \oplus \overline{\bigoplus_{\lambda > 0} \ker (\mathcal{L}- \lambda)} \]
 where $\mathcal{Q}$ is equipped with the Hilbert norm associated with the quadratic form $\mathcal{N}: x \mapsto q (x) + (1 - \lambda_0) \| x
  \|^2$.
\end{lemma}
\begin{remark} The hypothesis $\mathcal{L}$ has compact resolvent is equivalent to the compactness of the embedding $(\mathcal{Q},\mathcal{N}) \hookrightarrow \mathfrak{H}$.
\end{remark}
\begin{proof} Since $\lambda_n \rightarrow + \infty$, $\sum_{n \in
  \N} | \lambda_n | | \langle e_n, x \rangle |^2 < + \infty$ if and only if
  $\sum_{n \in \N} \lambda_n | \langle e_n, x \rangle |^2$ is convergent.
  
  \noindent By definition of $\mathcal{Q}$ (see \cite[théorème 3.10]{lewin:hal-03812036}), $\sqrt{\mathcal{N}}$
  is a Hilbert norm on $\mathcal{Q}$ and there is a continuous embedding $\mathcal{Q}
  \hookrightarrow \mathfrak{H}$ (since $\| x \|^2 \leqslant \mathcal{N}
  (x)$).
  
 \noindent For all $x \in \mathfrak{H}$, $x = \sum_{n \in \N} \langle e_n, x
  \rangle e_n$. Introduce $\mathcal{Q}' = \left\{ x \in \mathfrak{H}; \sum_{n \in
  \N} | \lambda_n | | \langle e_n, x \rangle |^2 < + \infty \right\}$. Assume $\sum_{n \in \N} | \lambda_n | | \langle e_n, x \rangle
  |^2 < + \infty$. Then, writing $x_N = \sum_{n = 0}^N \langle e_n, x
  \rangle e_n$, we have $x_N \in \mathcal{Q}$, $x_N \rightarrow x$ in
  $\mathfrak{H}$ and
  \[ q (x_{N + p} - x_N) = \sum_{n = N}^{N + p} \lambda_n |\langle e_n, x
     \rangle|^2  \] so $(x_N)$ is a Cauchy sequence for  $\sqrt{\mathcal{N}}$. We deduce $x \in
  \mathcal{Q}$ and $q (x) = \sum_{n \in \N} \lambda_n | \langle
  e_n, x \rangle |^2$. Therefore $\mathcal{Q}' \subset \mathcal{Q}$.
  
  \noindent If ${x \in \mathcal{Q}'}^{\bot_{\mathcal{N}}}$ then for all $n \in
  \N$, $x$ is orthogonal (for $\mathcal{N}$) to $e_n$ \textit{i.e.}   \[ (\lambda_n - \lambda_0 + 1) | \langle e_n, x \rangle |^2 = 0 \]
 so  $\langle e_n, x \rangle = 0$ and we conclude $x = 0$. It follows that
  $\mathcal{Q}'$ is dense in $\mathcal{Q}$ for the norm $\sqrt{\mathcal{N}}$. Let's show that it is also closed. If $(p_k)_{k \in
  \N}$ is a sequence of elements of $\mathcal{Q}'$ and $x \in
  \mathcal{Q}$ are such that $\mathcal{N} (p_k - x) \rightarrow 0$ then, we have in particular
  \[ \| p_k - x \|^2 \leqslant q (p_k - x) + (1 - \lambda_0) \| p_k - x \|^2
     \underset{k \rightarrow + \infty}{\rightarrow} 0 \]
  hence $p_k \rightarrow x$ in $\mathfrak{H}$ and $\mathcal{N} (p_k)
  \rightarrow \mathcal{N} (x)$ so $q (p_k) \rightarrow q (x)$. Additionally, for $N$ large enough,
  \[ q (p_k) = \sum_{n \in \N} \lambda_n | \langle e_n, p_k \rangle
     |^2 \geqslant \sum_{n = 0}^N \lambda_n | \langle e_n, p_k \rangle |^2 \]
  so
  \[ q (x) = \lim q (p_k) \geqslant \lim \sum_{n = 0}^N \lambda_n | \langle
     e_n, p_k \rangle |^2 = \sum_{n = 0}^N \lambda_n | \langle e_n, x \rangle
     |^2 \]
  and having $N$ go to $+ \infty$, we conclude that $x \in
  \mathcal{Q}'$.
\end{proof}

\begin{lemma} \label{indiceformequaddiago}
  With the notations of the previous lemma, we have
  \[ \dim \left( \bigoplus_{\lambda < 0} \ker (\mathcal{L}- \lambda) \right)
     = \mathrm{ind} q, \]
  \[ \ker \mathcal{L}= \ker q. \]
 Furthermore,
  \[ \mathrm{ind} q + \dim \ker q = \inf \{ \mathrm{codim} W \mid q_{|W} > 0 \} <
     + \infty . \]
  In particular, if $F$ is a subspace on which $q$ is negative semi-definite then $\dim F
  \leqslant \mathrm{ind} q + \dim \ker q$.
\end{lemma}

\begin{remark}
  Let $F, W$ two subspaces of a vector space $E$. Assume that $F
  \cap W = \{ 0 \}$. Then $\dim F \leqslant \mathrm{codim} W$. Indeed, the restriction to $F$ of the projection $E \rightarrow E / W$ is one-to-one.
\end{remark}

\begin{proof}
  Firstly, $\ker q = \ker \mathcal{L}$ and on $\displaystyle \bigoplus_{\lambda < 0}
  \ker (\mathcal{L}- \lambda)$, $q$ is negative-definite so
  \[ \dim \left( \bigoplus_{\lambda < 0} \ker (\mathcal{L}- \lambda) \right)
     \leqslant \mathrm{ind} q. \]
 %Furthermore, since the spectrum of $\mathcal{L}$ is bounded below, we have $\dim \left( \displaystyle\bigoplus_{\lambda \leqslant 0} \ker  (\mathcal{L}- \lambda) \right) < + \infty$.
  
  Let $W$ a subspace of $\mathcal{Q}$ such that $q_{|W} < 0$. Then $W \cap
  \overline{ \displaystyle\bigoplus_{\lambda \geqslant 0} \ker (\mathcal{L}- \lambda)} = \{
  0 \}$: if $w \in W \cap \overline{ \displaystyle \bigoplus_{\lambda \geqslant 0}
  \ker (\mathcal{L}- \lambda)}$, we can write $w = v + z$ with $v \bot
  z$, $v \in \ker q$ et $q (z) \geqslant 0$ so
  \[ q (w) = q (z) \geqslant \lambda_0 \| z \|^2 \]
 and we deduce $w = 0$. Thus,
  \[ \dim W \leqslant \mathrm{codim} \left( \overline{ \displaystyle \bigoplus_{\lambda
     \geqslant 0} \ker (\mathcal{L}- \lambda)} \right) = \dim \left(
     \bigoplus_{\lambda < 0} \ker (\mathcal{L}- \lambda) \right) \] which concludes the proof of the first equality.
    
  \noindent Let $W$ a subspace of $\mathcal{Q}$ such that $q_{|W} > 0$. Then $W \cap  \displaystyle
  \bigoplus_{\lambda \leqslant 0} \ker (\mathcal{L}- \lambda) = \{ 0 \}$ and
  $\mathrm{codim} W \geqslant \dim \left(   \displaystyle\bigoplus_{\lambda \leqslant 0} \ker
  (\mathcal{L}- \lambda) \right) = \ind q + \dim \ker q$. Furthermore $q$ is positive-definite on
  $W = \overline{ \displaystyle\bigoplus_{\lambda > 0} \ker (\mathcal{L}- \lambda)}$ and $\mathrm{codim} W =\ind q + \dim \ker q$, which proves the third equality.
  
  If a subspace $F$ is such that $q_{|F} < 0$, then for all subspace $W$ such that $q_{|W} > 0$ we have $W \cap F = \{ 0 \}$ so $\dim F \leqslant \mathrm{codim} W$. The last statement follows.
\end{proof}

\subsubsection{A diagonalisability condition}

\begin{lemma} \label{diagonalisationdelaformequad}
  Let $(\mathcal{M}^4, h)$ a closed Riemannian manifold, $E\to \mathcal{M}$ a $G$-bundle and $\nabla$ a Yang-Mills connection. Consider a weight function $\omega : \mathcal{M} \rightarrow \R_+$ such that $\omega$ and $1/\omega$ are in $\mathrm{L}^{\infty}(\mathcal{M})$.
 Define $\mathfrak{H} =\mathrm{L}^2 (\mathcal{M}, T^*M\otimes \mathfrak{g})$ equipped with the inner product associated to the norm \[ \Vert a \Vert^2_{\mathrm{L}^2_\omega} = \int_M |a|^2_h \omega \mathrm{vol}_h, \] and $W := \mathrm{W}^{1, 2} (\mathcal{M},T^*M\otimes \mathfrak{g})$. Fix a bounded operator $B : \mathfrak{H}\rightarrow \mathfrak{H}$ and assume that $B$ is symmetric for the usuel $\mathrm{L}^2$ inner product. 
  
  The operator $\omega^{- 1} ( \Delta_\nabla + B)$ on $\mathfrak{H}$ with domain $\Omega^1(\mathcal{M}, \mathfrak{g})$ has a unique self-adjoint extension with domain is a subset of $W$. Furthermore, this extension has compact resolvent and is bounded below.
\end{lemma}

\begin{remark}Since $\omega$ and $1/\omega$ are in $\mathrm{L}^\infty$, the $\mathrm{L}^2_\omega$ norm and the usual $\mathrm{L}^2$ norm are equivalent on  $\mathfrak{H}$.
\end{remark}

\begin{proof}  It is enough to show this result with $B = 0$ and then that $\omega^{- 1} B
  (\omega^{- 1}  \Delta_\nabla + 1)^{- 1}$ is a compact operator of $\mathfrak{H}$ (see \cite[corollaries 3.19, 5.24, theorem 5.31]{lewin:hal-03812036}).
  
    Introduce $(\mathcal{L}^{\min}, D (\mathcal{L}^{\min})) = \left(\omega^{- 1} \Delta_\nabla,
  \Omega^1(\mathcal{M}, \mathfrak{g})\right)$ and $Q$
  the corresponding quadratic form. If $a \in D (\mathcal{L}^{\min})$, \[ Q (a) = \langle a, \mathcal{L}^{\min} a \rangle_{\mathrm{L}^2_{\omega}} \geqslant 0\] so $Q$ is bounded below and closed on $W$. We define $(\mathcal{L}, D (\mathcal{L}))$ to be the Friedrichs extension of $\mathcal{L}^{\min}$. $\mathcal{L}$ is a non-negative self-adjoint operator and we have \[ D (\mathcal{L}) = \left\{ a \in W \mid \Delta_\nabla a \in \mathrm{L}^2\right\}= \mathrm{W}^{2, 2} (\mathcal{M}, T^*M\otimes \mathfrak{g}). \]Since the embedding $W\hookrightarrow \mathfrak{H}$ is compact by Rellich-Kondrachov theorem, $\mathcal{L}$ has compact resolvent.
  
  \noindent To conclude, since $\omega^{- 1} B : \mathfrak{H} \to \mathfrak{H}$ is bounded and $
  (\omega^{- 1}  \Delta_\nabla + 1)^{- 1} : \mathfrak{H} \to \mathfrak{H}$ is compact, $\omega^{- 1} B
  (\omega^{- 1}  \Delta_\nabla + 1)^{- 1}  : \mathfrak{H} \to \mathfrak{H}$ is compact, which concludes the proof.
\end{proof}
  
  \begin{comment} 
\subsection{Lemmes relatifs au poids}

\begin{lemma}
  \label{croissancepoids} Si $\eta' > \eta$ et $\lambda_k / \eta < \eta$, alors
  sur $M \backslash \mathrm{B}_{\lambda_k / \eta'} (p_k)$, $\omega_{\eta',
  k} \leqslant \omega_{\eta, k}$.
\end{lemma}

\begin{proof}
  On distingue plusieurs zones : $(I) = \{ | x | \geqslant \eta' \}$, $(I I) =
  \{ \eta \leqslant | x | \leqslant \eta' \}$, $(I I I) = \{ \lambda_k / \eta
  \leqslant | x | \leqslant \eta \}$ et $(I V) = \{ \lambda_k / \eta' \leqslant
  | x | \leqslant \lambda_k / \eta \}$.
  
  Dans $(I)$ et $(I I I)$ c'est une simple conséquence du fait que l'expression soit une fonction
  décroissante de $\eta$.
  
  Dans $(I I)$, $\omega_{\eta, k}$ est constante, $\omega_{\eta', k}$ est une
  fonction décroissante de $| x |$ donc 
  \begin{align*}
    \omega_{\eta', k} (x) & \leqslant  \eta^{- 2} \left( (\eta /
    \eta')^2 + (\eta \eta' / \lambda_k)^{-2} \right)\\*
    & \leqslant \eta^{- 2} (1 + (\eta^2 / \lambda_k)^{- 2} ) = \omega_{\eta, k} (x) .
  \end{align*}
  Dans $(I V)$, $\lambda_k / \eta \geqslant | x |$ et $\eta' | x | / \lambda_k
  \leqslant 1$
  \begin{align*}
       \omega_{\eta, k} (x) & \geqslant  | x |^{- 2} \left( (| x | /
       \eta)^{2} + \frac{(1 + \eta^2)^2}{(\eta^2 + (\eta | x | /
       \lambda_k)^2)^2}  \right)\\
       & \geqslant  | x |^{- 2} \left( (| x | / \eta')^{2} + 1\right)\\
       & \geqslant  | x |^{- 2} \left( (| x | / \eta')^{2} + (\eta' | x
       | / \lambda_k)^{- 2} \right) = \omega_{\eta', k} (x) .
  \end{align*}
     \end{proof}
 \end{comment}

\newpage

\addcontentsline{toc}{chapter}{Bibliography}

\bibliographystyle{plain}
\bibliography{references}

\end{document}